\documentclass{ems-icm-like-A4}

\usepackage[]{amsfonts,latexsym}
\newcommand{\remove}[1]{}
\usepackage{bbm}
\usepackage{verbatim}
\usepackage[ruled,linesnumbered]{algorithm2e}

\usepackage[fontsize=12pt]{fontsize}
\newcommand{\im}{\operatorname{Im}}

\makeatletter

\newcommand{\parens}[1]{\left( #1 \right)}
\newcommand{\sparens}[1]{\left[ #1 \right]}
\newcommand{\innerprod}[1]{\left< #1 \right>}

\newcommand{\set}[1]{\left\{ #1 \right\}}

\newcommand{\cochain}[1]{\MakeUppercase{#1}}
\newcommand{\cochainset}[2]{C^{#1}\parens{#2}}

\newcommand{\stdcochain}{\cochain{f}}
\newcommand{\gencochain}{\cochain{g}}

\DeclareMathOperator{\weightletter}{w}
\newcommand{\skeleton}[2]{#1^{\parens{#2}}}
\newcommand{\stdcomplex}{X}

\newcommand{\face}[1]{\MakeLowercase{#1}}
\newcommand{\vertex}[1]{\MakeLowercase{#1}}
\newcommand{\stdvertex}{\vertex{v}}
\newcommand{\genvertex}{\vertex{u}}
\newcommand{\stdface}{\face{\sigma}}
\newcommand{\genface}{\face{\tau}}
\NewDocumentCommand{\upoperator}{O{k}}{U^{#1}}
\NewDocumentCommand{\downoperator}{O{k}}{D^{#1}}
\NewDocumentCommand{\walkoperator}{O{k} o}{\IfNoValueTF{#2}{\Delta_{#1}}{\Delta_{#1, #2}}}
\NewDocumentCommand{\lapwalkoperator}{O{k} o}{\IfNoValueTF{#2}{\mathscr{L}_{#1}}{\mathscr{L}_{#1, #2}}}

\newcommand{\coboundaryoperator}{\delta}
\NewDocumentCommand{\weight}{O{k} o}{\IfNoValueTF{#2}{\weightletter \parens{ #1 }}{\weightletter_{#1}\parens{#2}}}

\newcommand{\indicator}[1]{\mathbbm{1}_{#1}}
\newcommand{\abs}[1]{\left| #1 \right|}
\newcommand{\norm}[1]{\lVert #1 \rVert}

\DeclareMathOperator{\ex}{\mathbbm{E}}
\newcommand{\ev}[2]{\ex_{#1}{\sparens{#2}}}

\newtheorem{theorem}{Theorem}[section]
\newtheorem{lemma}[theorem]{Lemma}
\newtheorem{definition}[theorem]{Definition}
\newtheorem{claim}[theorem]{Claim}

\newtheorem{conjecture}[theorem]{Conjecture}

\newtheorem{proposition}[theorem]{Proposition}
\newtheorem{corollary}[theorem]{Corollary}

\newcommand{\R}{\mathbb{R}}

\newtheorem{thm}{Theorem}[section]
\newcommand{\veps}{\varepsilon}

\newcommand{\suchthat}{\,:\,}
\newcommand{\where}{\,|\,}

\newcommand{\Ham}{Ham}

\NewDocumentCommand{\infullversion}{m m}{\IfBooleanTF{\BooleanFalse}{#1}{#2}}

\SetAlgoNlRelativeSize{-1}

\begin{document}

\volume{?} % Don't alter this line.

\title{No Where to Go But High: \\ A Perspective on High Dimensional Expanders}

\emsauthor{1}{Roy Gotlib}{R.~Gotlib}
\emsauthor{2}{Tali Kaufman}{T.~Kaufman}

\emsaffil{2}{Department of Computer Science, Bar-Ilan University, Ramat-Gan, 5290002, Israel\email{kaufmant@mit.edu}}
\emsaffil{1}{Department of Computer Science, Bar-Ilan University, Ramat-Gan, 5290002, Israel\email{roy.gotlib@gmail.com}}

\begin{abstract}
"No Where to go but in" is a well known statement of Osho. Osho meant to say that the answers to all our questions should be  obtained by looking into ourselves. In a paraphrase to Osho's statement we say "No Where to go but high". This meant to demonstrate that for various seemingly unrelated topics and questions the only way to get significant progress is via the prism of a new philosophy (new field) called high dimensional  expansion. In this note we give an introduction \footnote{This introduction reflects the authors' interests and by no mean claim to represent the field in a through way} to the high dimensional expansion philosophy, and how it has been useful recently in obtaining progress in various questions in seemingly unrelated fields.

\bigskip 

\begin{center}
This exposition is dedicated to the memory of my mother, Sarah Kaufman, who was always trying to understand the reason why things behave in a certain way. 

It is also dedicated to the memory of my father Eliezer Kaufman.
\end{center}
\end{abstract}

\maketitle

\newpage

\section{Introduction}

What is common to the following very diverse and important themes:  quantum codes, counting bases of matroids, locally testable classical codes, fast mixing of Markov chains and Gromov's topological overlapping question.
Even if you have not heard of some/any of these mentioned topics, it is clear that these topics emerge from completely different branches of mathematics and computer science, and hence do not seem particularly related.

The purpose of this note is to highlight the idea that all the above mentioned topics, despite seeming unrelated up until recently, are in fact strongly related via a new perspective that is obtained by introducing a new object called a \emph{high dimensional expander}.
This object created inherent, deep relations between topics that previously seemed unrelated. This new perspective, as well as new connections between different problems via the idea of high dimensional expanders, had recently led to various important advances on the above topics and beyond.

Our goal, in this note is \emph{not} to serve as a survey on high dimensional expanders, but rather to give \emph{our own} perspective on this newly emerged object, and its connections/strong implications to the above topics. Namely, to highlight recent advances on the said topics using the new perspective of high dimensional expanders.

We intend to highlight how this newly emerged object (or maybe newly emerged philosophy) is, in fact, tightly related to the above notions and beyond. High dimensional expansion has different angles and these different angles enable one to relate these diverse set of questions/phenomena.

Our aim here is to present what high dimensional expanders are and how these newly defined objects give a \emph{unified} perspective of the topics mentioned above that prior to the introduction of high dimensional expanders, seemed unrelated.

High dimensional expanders, as we will see, are generalizations of graph expanders to higher dimensions. But, as we will see, the importance of high dimensional expanders does \emph{not} stem from the fact that they generalize expander graphs to higher dimensions, but rather from the fact that when objects exhibit expansion in higher dimensions they also have strong form of \emph{local to global} properties/nature.

This local to global behaviour that high dimensional expanders exhibit is unique to the high dimensional case, in the sense that it is not present in one dimensional expanders.
Indeed, this local to global philosophy is what makes them so powerful and so connected to the various topics mentioned above.

Our focus here is not on rigorous proofs, but rather on presenting different angles of the high dimensional expansion philosophy, with and their recent implications for various unrelated fields.

\paragraph{Structure of this note.} The structure of this note will be as follows. We will start  by introducing basic facts about expander graphs, with specific focus on weighted graphs which are essential to the theory of high dimensional expanders (see Section~\ref{Sec:expander-grpahs}). Then, in Section~\ref{Sec:high -dim-expanders}, we will move to introduce the object which is the focus of this note: A high dimensional expander; We will highlight the fact that there is a notion of high dimensional topological expansion, which generalises the Cheeger constant for graphs, and high dimensional spectral expansion that generalizes the spectral expansion of graphs. In graphs the Cheeger inequality says that the topological expansion is, in a sense, equivalent to spectral expansion. In higher dimensions, however, a Cheeger inequality does \emph{not} hold (the spectral definition does not imply the topological definition and vise versa) . This demonstrates some of the deepness of the high dimensional expansion phenomenon.

We then turn to discuss high dimensional random walks (see Section \ref{sec:high-dim-RWs}). We define what are high dimensional random walks and show that if all links (i.e. local neighborhoods) of the high dimensional expander are expanding enough then high dimensional random walks mix rapidly.

We follow that discussion by a discussion of the local-to-global aspects of high dimensional expanders - aspects which are non existent in graph expansion.
The local-to-global implication will hold (via different proofs) for spectral expansion and for topological expansion.
We will show that a high dimensional simplicial complex whose all local links (i.e. local neighborhoods) are spectraly/topologically expanding must be a global spectral/topological high dimensional expander (see Section \ref{Sec:local-to-global-spectral-expansion} for local-to-global expansion in the spectral sense, and see Section \ref{Sec:local-to-global-topological-expansion} for local-to-global expansion in the topological sense).

Using the local-to-global premise and the fact that we have fast mixing of high dimensional random walks if all the links are sufficiently expanding, We will deduce fast mixing of high dimensional random walks from local expansion in very local neighborhoods.

Then, in Section \ref{Sec:LTCs}, we move to show that high dimensional expansion is a form of local testability of codes and use that towards results on local testability of codes.

We then, in Section \ref{Sec:QLDPC-codes}, demonstrate that local testability of classical codes and quantum LDPC codes are both \emph{born together} from a high dimensional expander.
Thus, we show that classical locally testable codes and quantum LDPC codes are connected together via the high dimensional expansion perspective.
Prior to the introduction of high dimensional expanders, these two objects were not known to be related.
We will survey some of the major recent developments in these fields that emerge from this recent viewpoint that connects them both to high dimensional expanders.

We then turn to discuss the Gromov Topological overlapping problem; its discovered relation to classical local testability and its solution via high dimensional expanders and their connection to locally testable codes (see Section \ref{sec:TOP}).  

We will end (see Section \ref{sec:sampling-matroids}) by introducing the Mihail-Vazirani Conjecture about counting bases of matroids that recently was resolved via high dimensional expansion.
We will introduce the conjecture and will show a \emph{rigorous} proof of its resolution via local to global theorems on high dimensional expanders and high dimensional random walks.

In the end of our note (see Section \ref{Sec:additional}) we mention some topics that are not covered by our note.
\section{Some Classical Facts About Expander Graphs}\label{Sec:expander-grpahs}

Before we discuss high dimensional expanders, we need a solid foundation in the original theory of expander graphs. Broadly speaking, there are two types of expanders: combinatorial expanders, and spectral expanders. The former generally refers to graphs $G=(V,E)$ where all subsets $S \subset V$ \textit{expand outward} in some sense (this includes definitions such as edge expansion, vertex expansion, or unique neighbor expansion). The latter is somewhat more of a technical condition: it requires that all ``non-trivial'' eigenvalues of $G$'s adjacency matrix be of bounded size. While a priori it is not obvious that these two notions of expansion are connected, it is well known that they are (at least morally) equivalent (see e.g. discussion of Cheeger's inequality and the expander-mixing lemma in \cite{hoory06}).

\subsection{Weighted Graphs}
In order to introduce expander graph in a way consistent with their higher dimensional counterparts we begin by introducing weighted graphs.
We assume that there are no isolated vertices and define weighted graphs as graphs that are equipped with a weight function for the edges that satisfies the following condition:
\begin{definition}[Weight function for the edges]
Let $G=(V,E)$.
A function $\weightletter_E:E \rightarrow (0,1]$ is a weight function for the edges if $\sum_{e \in E}{\weight[E][e]}=1$.
\end{definition}
The weight function for the edges of the graph induces a weight function over the vertices of the graph in the following way:
\begin{definition}[Weight function for the vertices]
Let $G=(V,E)$ and let $\weightletter_E$ be a weight function for the edges of the graph.
Then the following $\weightletter_V : V \rightarrow (0,1]$ is the induced weight function on the vertices of the graph:
\[
    \weight[V][v] = \sum_{\substack{e \in E \\ v \in e}}{\frac{1}{2}\weight[E][e]}
\]
\end{definition}
We note that this weight function over the vertices also sums up to $1$ (and therefore both weight functions define a probability distribution over the edges/vertices).
We can now use these two functions to define the weight function for the graph:
\begin{definition}[Weight function for a graph]
Let $G=(V,E)$ be a graph and let $\weightletter_E$ be a weight function over the edges of the graph.
Define the weight function over the graph $\weightletter:\set{\emptyset} \cup V \cup E \rightarrow \mathbbm{R}$ to be\infullversion{:
\[
\weight[\genface] = \begin{cases}
1 & \genface = \emptyset \\
\weight[V][\genface] & \genface \in V \\
\weight[E][\genface] & \genface \in E
\end{cases}
\]
}{ the function that satisfies $\weightletter(\emptyset)=1$, $\weightletter|_V=\weightletter_V$ and $\weightletter|_E=\weightletter_E$.}
\end{definition}
\subsection{Spectral Expansion of Graphs}\label{subsec:graphs:spectral-expansion}
Let us start by describing the spectral notion of expansion.
We will start by going over some basic spectral properties of weighted graphs.
To start, we need to define the main object of interest, the adjacency matrix.
\begin{definition}[Adjacency Matrix]
The adjacency matrix of a weighted graph $G=(V,E)$ with a weight function $\weightletter$ is:
\[
A_{v,w} = \frac{\weight[\set{v,w}]}{2 \weight[v]}
\]
\end{definition}
One natural (and useful) way to think about the adjacency matrix is as the transition matrix\footnote{Technically the transpose, as we'll discuss later.} of the random walk underlying $G$ which moves from a vertex $v$ to vertex $w$ with probability corresponding to the edge weights incident to $v$. Indeed, much of this survey will focus on standard connections between properties of this walk and the eigenvalues of $A$, and how they generalize to higher dimensions. First, however, we need a few basic spectral facts about the adjacency matrix $A$.
\begin{proposition}
Let $G=(V,E)$ be a weighted graph, and $A_G$ its corresponding adjacency matrix. Then $A_G$ has a spectral decomposition (i.e.\ an orthonormal basis of eigenvectors).
\end{proposition}
\begin{proof}
This follows from the Cauchy's celebrated spectral theorem: that any self-adjoint operator can be diagonalized.
The trick is then to find an inner-product space over which $A_G$ is self-adjoint.
One can show that this is the case for the standard inner product normalized by the distribution $\weightletter$ induces over vertices.
For $f,g: V \to \mathbbm{R}$, let
\[
\langle f,g \rangle = \sum\limits_{v \in V}\weight[v]f(v)g(v)
\]
A simple computation verifies that $\langle Af,g \rangle = \langle f, Ag \rangle$.
\end{proof}
Now that we know $A_G$ has a well-defined spectrum we can examine properties of its eigenvalues. For a graph $G$ on $n$ vertices, denote the eigenvalues of $A_G$ in decreasing order by $\lambda_1 \geq \ldots \geq \lambda_n$. For our purposes, we are most interested in two basic properties of these eigenvalues.
\begin{claim}
Let $G=(V,E)$ be a weighted graph, and $A_G$ its corresponding adjacency matrix with eigenvalues $\lambda_1 \geq \ldots \geq \lambda_n$. Then (1) The all $1$'s vector is an eigenvector satisfying $A_G \mathbbm{1} = \mathbbm{1}$ and (2) This is also the maximal eigenvalue (in absolute value): \\$1= \lambda_1 \geq \lambda_2 \ldots \geq \lambda_n \geq -1$
\end{claim}
It is well known that a graph $G$ is \emph{connected} if and only if $\lambda_2 < 1$.
Spectral expansion studies a strengthening of this condition, when the second eigenvalue is \emph{bounded away from $1$} by some constant.
\begin{definition}[Weighted Spectral Expansion]
We say a weighted graph $G=(V,E)$ is a (one-sided) $\lambda$-spectral expander if $\lambda_2 \leq \lambda$.
The quantity $1-\lambda_2$ is often called the \emph{spectral gap}, and is an equivalent way to express one-sided spectral expansion.

\end{definition}
At the beginning of this section, we mentioned that spectral expansion is intimately tied to combinatorial expansion. Intuitively, this follows from the fact that the latter can be viewed as a sort of robust connectivity.
Thought of in this manner, it is no surprise that since $\lambda_2 < 1$ enforces that $G$ is connected, smaller $\lambda_2$ corresponds to stronger notions of connectivity akin to combinatorial expansion.

\subsection{Toplogical expansion of graphs}
Informally, a finite graph is called an expander if relatively many edges cross between every set of vertices and its complement.

If $G$ is a weighted graph with $\weightletter$ as its weight function (one can take $w_E = \frac{1}{\abs{E}}$ to ignore the weights).
Then one can define the following norm over sets of vertices/edges (or the indicator of such sets):
\[ \norm{S}_w = \sum_{s \in S}{w(s)}, \quad \norm{\indicator{S}}_w = \norm{S}_w \]
In this case, the topological expansion of $(G,w)$ is quantified by its \emph{Cheeger constant}, defined as
\[h(G,w) = \min_{\emptyset\neq S\subsetneq V}{\frac{\norm{\delta \indicator{S}}_w}{\min{\set{\norm{\indicator{S}}_w, \norm{\indicator{V \setminus S}}_w}}}}\]
Where $\delta \indicator{S}$ is a function that accepts an edge $\set{u,v}$ and returns $1$ iff $\indicator{S}(u)+\indicator{S}(v) = 1$ (the sum is performed modulo 2). 
Note that the numerator of the Cheeger constant is the norm of the edges that connect $S$ and $V \setminus S$. One says that $(G,w)$ is an $\veps$-expander if $h(G, \weightletter)\geq \veps$.

We know that if a weighted graph $(G,w)$ is a $\lambda$-spectral expander, then $(G,w)$ is has a good Cheeger constant.
We call such inequalities Cheeger ineqalities and in the following establish a weighted version of the Cheeger Inequality.
Recall that given a weighted graph $(G,w)$ and $A \subseteq V$, we write $\indicator{A}$ for the indicator function of $A$.

\begin{thm}[Weighted Cheeger Inequality]
	\label{TH:Cheeger}
	Let $(G,\weightletter)$ be a weighted graph which is also a $\lambda$-spectral-expander and let $A\subseteq V$.
	Then due to~\cite[Theorem 4.4]{kaufman2021high}:
	\[
	\norm{\delta \indicator{A}}_{\weightletter} \ge 2 \parens{1-\lambda}\norm{\indicator{A}}_{\weightletter}\parens{1-\norm{\indicator{A}}_{\weightletter}}
	\]
	In addition, due to~\cite[Theorem 2.1]{friedland2002cheeger}, if $h\parens{G,\weightletter}\ge \epsilon$ then $(G, \weightletter)$ is a $\sqrt{1-\frac{\epsilon^2}{4}}$-spectral-expander.
\end{thm}
Importantly we have not analogue for the Cheeger inequality in higher dimensions; i.e. high dimensional spectral expansion does not imply high dimensional topological expansion, and implication in the other direction also does not hold.
\section{High Dimensional Expanders: The Object of Study}\label{Sec:high -dim-expanders}
\subsection{Simplicial Complexes and Links}
In order to generalize an expander graph to higher dimensions we first have to define an object that has higher dimension.
Our object of choice is a \emph{pure simplicial complex}.
These are objects that generalize graphs in two important manners:
Firstly, much like graphs only contain an edge if both vertices that are contained in it are in the graph. If a simplicial complex contains a higher dimensional edge then it contains all of the lower dimensional edges that are contained in it.
Secondly, the complex does not include a high dimensional isolated vertex, i.e.\ every edge is contained in a maximal edge.
\begin{definition}[Pure simplicial complex]
A collection of sets $\stdcomplex$ is a pure simplicial complex if it satisfies the following:
\begin{itemize}
    \item \emph{$\stdcomplex$ is a simplicial complex:} If $\stdface \in \stdcomplex$ and $\genface \subset \stdface$ then $\genface \in \stdcomplex$.
    \item \emph{$\stdcomplex$ is pure:} If $\stdface, \genface \in \stdcomplex$ are maximal sets (a set is \emph{maximal} if it is not contained in any strictly larger set of the complex) in $\stdcomplex$ then $\abs{\stdface}=\abs{\genface}$.
\end{itemize}
We call subsets in $\stdcomplex$ the \emph{faces} of $\stdcomplex$.
\end{definition}
We define the dimension of a face as:
\begin{definition}[Dimensions]
Given a face $\stdface \in \stdcomplex$ we define the dimension of $\stdface$ to be $\dim\parens{\stdface} = \abs{\stdface}-1$ and the dimension of the complex to be the dimension of the maximal face in $\stdcomplex$.
We also define the set of faces of a certain dimension in the following way:
\[
    \stdcomplex(i) = \set{\stdface \in \stdcomplex \suchthat \dim\parens{\stdface}=i}
\]
We should also stress that $\stdcomplex\parens{-1} = \set{\emptyset}$.
\end{definition}

\begin{definition}[Degree]
The degree of a simplicial complex is the maximal number of faces on a single vertex. A family of simplicial complexes with growing number of vertices, is said to have a \emph{bounded degree} if their degree is independent on the number of vertices, and remains fixed as the number of vertices in the family grows.  
\end{definition}

We note that one can think of pure $d$-dimensional complex as a $d$-uniform hypergraph with closure property.

We often refer to local neighborhoods of a simplicial complex. These are called links. Links are playing a major role in studying high dimensional expanders, as much of the study is done via the local to global paradigm, where we study complex by its links. Links are defined as follows.

\begin{definition}[Links]
Let $X$ be a $d$-dimensional pure simplicial complex.
For every $i$ and $\tau \in X(i)$, the \emph{link} of $\tau$ is the restriction of the complex to faces containing $\tau$, that is:
\[
X_\tau = \set{ \sigma \setminus \tau :~ \sigma \in X \text{ and } \tau \subseteq \sigma}.
\]
Put into words, $X_\tau$ is the complex that arises by selecting all $d$-faces that contain $\tau$, then removing $\tau$ itself. Finally, it is often convenient to refer to the set of links for all $\tau \in X(i)$, which we will refer to as the \textbf{$i$-links}.
\end{definition}

Links provide a natural method for decomposing global functions on simplicial complexes into local parts.
For instance, it is not hard to see that given a function on $k$-faces $f: X(k) \to \mathbb{R}$, its expectation over the complex is equal to the average of its expectation over links:\footnote{Formally, these expectations are defined over \textit{weighted} complexes where each level (and link) is endowed with a distribution. We cover this in the following section.}
\[
\underset{X(k)}{\mathbb{E}}[f] = \underset{\tau \in X(i)}{\mathbb{E}}\left[ \underset{\sigma \in X_\tau(k-i)}{\mathbb{E}}[f(\tau \cup \sigma)] \right].
\]

\subsection{Weighted Simplicial Complexes and Weighted Links}

When defining high dimensional expanders we would need to work with \emph{weighted} simplicial complexes.
\begin{definition}[Weighted Simplicial Complex]
A weighted pure $d$-dimensional simplicial complex $(X,\Pi)$ is a pure $d$-dimensional simplicial complex $X$ endowed with a distribution $\Pi$ on faces of the maximal dimension.
The weight of a $k$-dimensional face in the simplicial complex is then defined in the following way:
\[
\weight[\genface] = \begin{cases}
\Pi(\genface) & \genface \in k=d\\
\frac{1}{\abs{\genface}+1}\sum_{\substack{\stdface \in X(k+1)\\ \genface \subseteq \stdface}}{\weight[\stdface]} & \text{otherwise}
\end{cases}
\]
\infullversion{
Note that the weight of a $k$-dimensional face corresponds to the probability that it is chosen by the following process:\\
\begin{algorithm}[H]
\caption{Pick a face distributed according to the norm}
\DontPrintSemicolon
Draw a $d$-face $\sigma$ distributed according to $\Pi$.\;
\While{$\abs{\sigma} > k+1$}{
    Pick uniformly $v \in \sigma$.\;
    Set $\sigma = \sigma \setminus v$\;
}
\Return{$\sigma$}.\;
\end{algorithm}
Equivalently, one can pick the vertices to remove simultaneously and therefore:
\[
\forall \tau \in X(i): \weight[\tau] = \frac{1}{\binom{d}{i}}\sum_{\substack{\sigma \in X(d) \\ \tau \subset \sigma}}{\Pi(\sigma)}
\]
}{}
When $\Pi$ is not specified, it is assumed to be uniform.
\end{definition}
The \emph{weighted} links of a weighted complex are, naturally, themselves weighted complexes with distributions inherited from the global distribution $\Pi$.
\begin{definition}[Weighted Links]
Let $(X,\Pi)$ be a $d$-dimensional weighted simplicial complex. For all $0 \leq i \leq d$ and $\tau \in X(i)$, the weighted link $\parens{X_\tau,\Pi_{\tau}}$ is given by:
\begin{enumerate}
\item $X_\tau = \set{ \sigma \setminus \tau :~ \sigma \in X \text{ and } \tau \subseteq \sigma}$
\item $\weight[\tau][\sigma] = \Pr_{\sigma' \sim \Pi}{\sparens{\sigma' = \sigma \mid \tau \subseteq \sigma'}} =\frac{\weight[\sigma \cup \tau]}{\binom{\abs{\sigma}+\abs{\tau}}{\abs{\tau}}\weight[\tau]}$
\end{enumerate}
In other words, the distribution over $X_\tau$ is simply given by normalizing $\Pi$ over the top level faces of $X_\tau$. 
Finally, note that we usually drop the distribution $\Pi_{\tau}$ when clear from context.
\end{definition}

Much like in the one dimensional case, we will be interested in defining a norm on sets of faces of some dimension.
And, again, much like the graph case we will do so in the following way:
\begin{definition}[Norm]
For any weighted pure $d$-dimensional simplicial complex and every dimension $i$, define the following norm:
\[
    \forall S \subseteq X(i): \norm{S}_w = \sum_{s \in S}{w(s)}, \quad \norm{\indicator{S}}_w = \norm{S}_w 
\]
\end{definition}

\subsection{Spectral Definition of High Dimensional Expansion}
We are new ready to give definitions of a high dimensional expanders. We will give spectral and topological definitions of high dimensional expanders. The first definition we are going to consider is a spectral definition of high dimensional expanders that are called  \emph{local spectral expanders}.
These are simplicial complexes whose links are excellent expanders in the sense that their underlying graph is an expander.
Formally consider the following two definitions:
\begin{definition}[Skeleton]
    Let $X$ be a simplicial complex define the $i$ skeleton of $X$ to be the following simplicial complex:
    \[
        \skeleton{X}{i} \coloneqq \set{\stdface \in X \suchthat \dim{\stdface} \le i}
    \]

\end{definition}
Using this definition we consider the underlying graph of the links as their $1$-skeleton and arrive at the following definition that was introduced by~\cite{kaufman2016isoperimetric,EvraK16, oppenheim2018local}.
\begin{definition}[Local spectral expander~\cite{ kaufman2016isoperimetric,EvraK16, oppenheim2018local}]
    A $d$-dimensional complex $X$ is a $\lambda$-local spectral expander if for every $i \le d-2$ and $\tau \in \stdcomplex(i)$ it holds that:
    \begin{itemize}
        \item $X_\tau^{(1)}$ is connected.
        \item $\lambda_2\parens{\skeleton{\stdcomplex_\genface}{1}} \le \lambda$.
    \end{itemize}
    $X$ will be called \emph{strong}-local-spectral expander if it is a $\lambda$-local spectral expander with $\lambda < \frac{1}{d}$. Otherwise it is called \emph{weak}-local-spectral expander.
\end{definition}
Note that in this definition we only regarded some of the links.
This is because the rest of the links are either a set of unconnected vertices or a complex that contains only the empty face.

\subsection{Topological Definition of High Dimensional Expansion}
Another generalization of expander graphs to higher dimensions generalizes them in a topological sense.
Consider the Cheeger constant definition of expansion in the graph case:
The Cheeger constant is the proportion between the weight of edges that "go out" of the set and the weight of the set itself (for sufficiently small sets).
In case of higher dimension, it will be useful to think of the indicator function of a set of sets. We therefore define:
\begin{definition}[cochains]
For a weighted simplicial complex $(X,w)$ define the set of cochains of $X$ over an abelian group $G$ to be:
\[
    C^i(X;G) \coloneqq G^{X(i)}
\]
\end{definition}
For $G =\mathbbm{F}_2$, this definition coincides with an indicator function for a set of sets of size $i+1$. For the vast majority of this note the cochains will be, indeed, defined over $\mathbbm{F}_2$.
We will therefore state explicitly when we are using a different underlying group.
Moreover, the rest of the definitions in this section are out of the scope of this note for cochains that are defined over groups that are not $\mathbbm{F}_2$. Therefore we will only present the following definitions over $\mathbbm{F}_2$.

It is then natural to ask how to define, for example, a triangle leaving a set of edges (similar to an edge leaving a set of vertices).
Consider the following generalization of $\delta$:
\begin{definition}[coboundary operator]
The \emph{coboundary operator} \\
$\delta_i: \cochainset{i}{\stdcomplex;\mathbbm{F}_2} \rightarrow \cochainset{i+1}{\stdcomplex;\mathbbm{F}_2}$ is the following operator:
\begin{equation}\label{EQ:classical-coboundary}
    \delta_i F (\sigma) = \sum_{\tau \in \binom{\sigma}{i+1}}{F(\tau)}
\end{equation}
In most cases the dimension will be clear from context and therefore omitted.
\end{definition}
Therefore we say that a triangle is leaving a set of edges if an odd number of its edges is in the set.
In addition, a standard computation shows that $\delta_{i+1}\circ \delta_i=0$.
We can therefore define the spaces of \emph{$i$-coboundaries} and the space of \emph{$i$-cocycles} as
\begin{equation}\label{EQ:Bi-Zi-def}
B^i=B^i(X)=B^i(X;\mathbbm{F}_2)=\im \parens{\coboundaryoperator_{i-1}}
\qquad\text{and} 
\qquad
    Z^i=Z^i(X)=Z^i(X;\mathbbm{F}_2)=\ker \parens{\coboundaryoperator_i},
\end{equation}
respectively, where $\delta_{-2}=0$ by convention.
We have $B^i\subseteq Z^i\subseteq C^i$ because $\coboundaryoperator_i\circ \coboundaryoperator_{i-1}=0$, and the quotient space $H^i(X;\mathbbm{F}_2) = Z^i/B^i$ is the \emph{$i$-th cohomology} space.
The space dual to $H^i(X;\mathbbm{F}_2)$ is the \emph{$i$-th homology} space denoted as $H_i(X;\mathbbm{F}_2)$.
We say that $X$ is $i$-dimensional \emph{$\mathbbm{F}_2$-connected} if $H^i(X;\mathbbm{F}_2)=0$.

We would now like to move on to describe the high dimensional generalization of the Cheeger constant.
Before we do that, however, we have to thoroughly inspect the denominator of the Cheeger constant.
Note that in any graph there are sets that are guaranteed not to expand.
These sets are the empty set and the whole graph.
Therefore, in the Cheeger constant we are not looking for the absolute expansion of set but we relate the expansion of the set to how different it is from one of these trivially non-expanding sets.
In the high dimensional case we very much do the same.
Here, however, there will be more sets that are trivially non-expanding.
Specifically, these sets are the sets that are coming from a lower dimension.
Formally, a set of the form $\delta F$ where $F$ is an $(i-1)$-dimensional set is trivially non-expanding in the $i$-th dimension.
These sets are called the coboundaries of $X$ and are denoted as $B^i(X)$.
We therefore define coboundary expansion, the higher dimensional analogue of the edge expansion in the following way. The definition is originated in the work of Linial and Meshulam and the work of Gromov~\cite{LM06, Gro10}.
\begin{definition}[Coboundary Expansion,~\cite{LM06, Gro10}]
Let $(X,w)$ be a pure, $d$-dimensional \\weighted simplicial complex.
Define the following generalization of the Cheeger constant:
\[
 h^i(X,w) = \min_{F \in C^i(X) \setminus B^i(X)}{\set{\frac{\norm{\delta F}_w}{\min_{G \in B^i(X)}{\set{\norm{F+G}_w}}}}}
\]
We say that a simplicial complex is an $\varepsilon$-coboundary expander if $h^i(X,w) \ge \epsilon$ for every dimension.
\end{definition}
Note here that $h^0(X,w)$ is the Cheeger constant of the graph corresponding to the one skeleton of the complex. 
Note further that $h^i(X,w) > 0$ iff $H^i(X;\mathbbm{F}_2)=0$.

To date, no known bounded degree coboundary expanders are known.
However, in most cases a relaxation of this condition suffices, namely: cosystolic expansion.
Cosystolic expanders are the high dimensional analogues of graphs with several large connected components that are all expanders.
The high dimensional analogue very much follows suite:
The high dimensional analogue of connected component is a generalization of the fact that a connected component is a set of vertices that has no outgoing edges which are, in fact, the cocycles. 
Therefore a complex with large connected components that are all expanders has been defined by~\cite{kaufman2016isoperimetric,EvraK16} as follows. 
\begin{definition}[Cosystolic expansion~\cite{kaufman2016isoperimetric,EvraK16}]
Let $(X,w)$ be a pure, $d$-dimensional simplicial complex.
Consider the following two definitions:
\begin{enumerate}
    \item \textbf{The expansion of the connected components:}
    \[
        \tilde{h^i}(X,w) = \min_{F \in C^i(X) \setminus Z^i(X)}{\set{\frac{\norm{\delta F}_w}{\min_{G \in Z^i(X)}{\set{\norm{F+G}_w}}}}}
    \]
    \item \textbf{The minimal connected component:}
    \[
        cosyst^i(X,w) = \min_{F \in Z^i(X) \setminus B^i(X)}{\set{\norm{F}_w}}
    \]
\end{enumerate}
We say that a weighted simplicial complex is an $(\epsilon, \mu)$-cosystolic expander if for every dimension it holds that both all the connected components are expanding, i.e. $\tilde{h^i}(X,w) \ge \epsilon$ and all of the connected components are large $cosyst^i(X,w) \ge \mu$.
\end{definition}

Now that we have generalized expansion to higher dimensions let us discuss some of the properties of high dimensional expanders.
We will start with considering one of the most important properties of high dimensioanl expanders, namely, the fact that high dimensional random walks converge rapidly to their stationary distribution.

\section{Random Walks on High Dimensional Expanders}\label{sec:high-dim-RWs}
High dimensional random walks are one of the major tools in the study of high dimensional expanders.
They were introduced by the work of Kaufman and Mass \cite{KaufmanM17} and were further developed by \cite{dinur2017high, kaufman2018high, alev2020improved, kaufman2021local}.
They have various implications, however, in this note we are going to mention only few.
One of the most important properties of local spectral expanders is the fast convergence of random walks to their stationary distribution.

The proof of fast mixing of random walks on high dimensional expanders is in the spirit of the local to global method developed by Garland.
Garland \cite{garland1973p} has shown that a simplicial complex whose all links spectrally expand, has vanishing cohomology over $\mathbb{R}$.
One can think about Garland's result as studying expansion over $\R$ using local to global arguments.
Essential to Garland's method is the fact that over $\mathbb{R}$ one can use self-adjoint operators and inner products which do not exist in all spaces, for example: over $\mathbb{F}_2$.

We use a method similar in spirit Garland to prove fast mixing of high dimensional random walks.  

Unlike the graph case, in which there is one canonical random walk, in a high dimensional expander many random walks are considered.
We will be interested in the \emph{up-down walk} and the \emph{down-up walk} of every dimension.
The $k$\textsuperscript{th} dimensional up-down walk transitions between faces of dimension $k$.
Every step of the walks is comprised of two sub-steps - the up step and the down step.
If at the beginning of the step the walk is at $\stdface$ then in the up step, a $(k+1)$-dimensional face $\genface$ is chosen that contains $\stdface$ with distribution proportional to its weight.
Then in the down step, a $k$-dimensional face that is contained in $\genface$ with equal probability.
The down-up walk can similarly be defined as taking the step down first and then taking the step up.
Formally these are defined as:
\begin{definition}[Down and up walks]
    Define the up walk as $\upoperator[k]:\cochainset{k}{\stdcomplex; \R} \rightarrow \cochainset{k}{\stdcomplex; \R}$ as:
    \[
        \upoperator[k]\stdcochain(\genface) = \ev{\stdface \in \stdcomplex(k)}{\stdcochain(\stdface) \mid \stdface \subseteq \genface}
    \]
    And the down walk $\downoperator[k]:\cochainset{k}{\stdcomplex; \R} \rightarrow \cochainset{k-1}{\stdcomplex; \R}$ as:
    \[
        \downoperator[k]\stdcochain(\genface) = \ev{\stdface \in \stdcomplex(k)}{\stdcochain(\stdface) \mid \genface \subseteq \stdface}
    \]
\end{definition}
And the corresponding random walks were defined as:
\begin{definition}[Up-down random walk and down up random walk]
    Define the up-down random walk and the down-up random walk as:
    \[
    \walkoperator[k]^{+}=\downoperator[k+1]\upoperator[k]\text{\qquad and \qquad} \walkoperator[k]^{-}=\upoperator[k-1]\downoperator[k]
    \]
    respectively.
\end{definition}
As previously mentioned, the convergence of these random walks is a key property of high dimensional expanders.
Since their introduction by Kaufman and Mass in~\cite{KaufmanM17}, high dimensional random walks have proven themselves to be the backbone of many of the applications of high dimensional expander in computer science.
Examples include:
Resolution~\cite{anari2019log} of the Mihail-Vazirani conjecture~\cite{mihail1989expansion}. (which we will rigorously present in Section~\ref{sec:sampling-matroids}), deandomization of direct product testing~\cite{dinur2017high,dikstein2019agreement,KaufmanM20} and more.

We say that these random walks converge rapidly due to:
\begin{theorem}[Random walks converge rapidly on good enough local spectral expanders~\cite{kaufman2018high}]\label{thm:ko-random-walks}
    If $\stdcomplex$ is a $\gamma$-local spectral expander then:
    \[
        \lambda_2\parens{\walkoperator[k-1]^{+}}=\lambda_2\parens{\walkoperator[k]^{-}} \le 1-\frac{1}{k+1}+\frac{k}{2}\gamma
    \]
\end{theorem}

Note that Theorem~\ref{thm:ko-random-walks} yields non trivial results for every dimension only when $\gamma \in O\parens{\frac{1}{d^2}}$ which suffices for the many applications.
There are, however, cases in which we might be interested in convergence of random walks and will not have such strong expansion assumptions (for example when trying to sample independent set with the hardcore distribution~\cite{anari2021spectral}).
In~\cite{alev2020improved} Alev and Lau relaxed this requirement and prove the following theorem:
\begin{theorem}[Random walks converge rapidly even for weak local spectral expanders~\cite{alev2020improved}]\label{thm:al-random-walks}
    Let $\stdcomplex$ be a pure $d$-dimensional simplicial complex and let
    \[
        \gamma_i = \max\set{\lambda_2\parens{\stdcomplex_\genface^{(1)}} \suchthat \genface \in \stdcomplex(i)}
    \]
    Then:
    \[
    \lambda_2\parens{\walkoperator[k]^{-}} \le 1-\frac{\prod_{i=-1}^{k-2}{\parens{1-\gamma_i}}}{k+1}
    \]
\end{theorem}
Which is meaningful whenever all the $1$-skeleton of the links is connected.

It is important to note that both in Theorem~\ref{thm:ko-random-walks} and Theorem~\ref{thm:al-random-walks} the key observation is that the $k$-dimensional random walks can be viewed through the random walks over the links of the complex.
For example, consider a step in the up-down random walk.
This step corresponds to first picking a link to walk over (specifically the link of a $k-1$ dimensional face).
In that link the original face is a vertex.
Then the walk simply picks one of its neighbours and walks there.

Since the original submission of this note, a new result by Gotlib and Kaufman~\cite{https://doi.org/10.48550/arxiv.2208.03241} managed to improve Theorem~\ref{thm:al-random-walks} further by proving a convergence result that takes into account the structure of the initial state.
In addition, their work managed to unify this result with the trickling down Theorem (Theorem~\ref{thm:trickling-down-theorem}).

Recent works on high dimensional random walks are interested in stronger property, namely they care for optimal mixing time and have studied (see e.g. \cite{cryan2019modified}) conditions under which such optimal mixing time holds.

\section{Local to Global Spectral Expansion of High Dimensional Expanders}\label{Sec:local-to-global-spectral-expansion}
In this section we present a \emph{local to global} property of spectral high dimensional expanders.
Specifically, we will show that if all the links are connected, global expansion can be derived from local expansion. The philosophy implemented in the proof is similar in spirit to that of Garland, where one uses self adjoint operators and inner products defined over $\mathbbm{R}$ to study expansion over $\mathbbm{R}$. 

However, the local-to-global descent of spectral gaps is \emph{inherently} different than the random walks result in the following way. The fast mixing of random walks result assumes spectral expansion in all links including the link of the empty set (which is a global expansion condition on the whole complex!) to conclude fast mixing of random walks. Hence that result is not based only on local assumptions. In contrast, the following result about descent of spectral gaps from links to the entire complex only assumes spectral expansion in \emph{local} links to deduce global expansion! So the following result is a \emph{genuine} local to global result, while the random walk result also assumed some global property.

We emphasise, however, that the ability to get the descent of spectral expansion from the links to the global object requires the complex to be \emph{strong-local-spectral expander}.
By that we mean that the complex is $\gamma$-local-spectral expander with $\gamma < \frac{1}{d}$, where $d$ is the dimension of the complex.
Namely, the descent of spectral gaps that we are going to present, and is known under the name of "Trickling Down Theorem" is only possible under a strong local spectral expansion guaranty.
This is in contrast with the previously discussed fast mixing of random walks in local spectral expanders, which we have shown to hold for any local spectral expander, not necessarily strong (see, e.g.,~\cite{golowich2021improved} for an example of a local spectral expander that is not strong).
The point we are making is that strong local spectral expansion is inherently different than non-strong local spectral expansion, as both allow for fast mixing of random walks, but only the strong one exhibits trickling down effects. 
We further note that all currently known combinatorial constructions of high dimensional expanders yield only weak spectral expanders. 
Constructions of strong spectral expanders are only known via algebraic means.

\begin{theorem}[Trickling Down Theorem~\cite{oppenheim2018local}]\label{thm:trickling-down-theorem}
    Let $\stdcomplex$ be a pure $n$-dimensional simplicial complexes with connected links.
    If, for every vertex $v$ it holds that $X_v$ is a $\lambda$-local spectral expander then $X$ is a $\frac{\lambda}{1-\lambda}$-local spectral expander.
\end{theorem}
We will prove the theorem by looking at the Laplacian of the adjacency matrix.
We will denote the adjacency matrix of the complex by $A$ and define the following:
\begin{definition}
    Let $\lapwalkoperator[0]^{+}$ be the Laplacian of up-down walk on the complex $X$.
    I.e.: $\lapwalkoperator[0]^{+} = I - A$.
    In addition, denote Laplacian of the link of $\stdface$ by $\lapwalkoperator[\stdface][0]^{+}$.
\end{definition}
We note that $\lapwalkoperator[0]^{+}$ is self adjoint and define the following:
\begin{definition}[Restriction]
    Let $\stdface$ be a face and $\stdcochain \in \cochainset{k}{\stdcomplex ; \mathbbm{R}}$ be a cochain.
    define the restriction of $\stdcochain$ to $\stdface$ to be $\stdcochain^{\stdface} \in \cochainset{k}{\stdcomplex_\stdface; \mathbbm{R}}$\infullversion{:
    \[
        \stdcochain^{\stdface}(\genface) = \stdcochain(\genface)
    \]}{ such that: $\stdcochain^{\stdface}(\genface) = \stdcochain(\genface)$.}
\end{definition}
\begin{lemma}\label{lem:average-over-links}\label{cor:averaging-over-co-boundary}
    Let $\stdcochain, \gencochain \in \cochainset{k}{\stdcomplex ; \mathbbm{R}}$ and let $0 \le l \le n-k-1$.
    Then:
    \[
        \innerprod{\stdcochain, \gencochain} = \ex_{\genface \in \stdcomplex(l)}{\sparens{\innerprod{\stdcochain^\genface, \gencochain^\genface}}}
    \]
    In addition, if $\stdcochain, \gencochain \in \cochainset{0}{\stdcomplex ; \mathbbm{R}}$ then:
    \[
        \innerprod{\lapwalkoperator[0]^{+} \stdcochain, \gencochain} = 
        \ex_{\stdface \in \stdcomplex(l)}{\sparens{\innerprod{\lapwalkoperator[\stdface][0]^{+} \stdcochain^\stdface, \gencochain^\stdface}_{\stdface}}}
    \]
\end{lemma}
This Lemma is the key of the local to global argument: 
We are interested in studying the behaviour of some process over the complex.
In this example - the up-down random walk.
In order to do so we look at the process from a local point of view, i.e.\ the links of the complex.
We then use properties we already know about the links in order to derive that the entire complex satisfies the property as well.
We are interested in the smallest non-zero eigenvalue of $\lapwalkoperator[0]^{+}$.
It would therefore be useful to understand the eigenspace whose eigenvalue is exactly $0$.
We would therefore consider the following projection into the eigenspace of $0$:
\begin{definition}
    $\lapwalkoperator[\stdvertex][0]^{-}$ is the projection to the space of constant functions, formally:
    \[
        \forall \stdcochain \in \cochainset{0}{\stdcomplex;\mathbbm{R}}: \lapwalkoperator[0]^{-} \stdcochain (\stdvertex) = \innerprod{\stdcochain, \mathbbm{1}}\mathbbm{1} = \ex_{\genvertex \in \stdcomplex(0)}{\sparens{\stdcochain(\genvertex)}}
    \]
    And, as with $\lapwalkoperator[0]^{-}$ we also define local versions of this operator:
    \[
    \lapwalkoperator[\genvertex][0]^{-} \stdcochain (\stdvertex) = \innerprod{\stdcochain, \mathbbm{1}_\genvertex}_\genvertex\mathbbm{1}_\genvertex = \ex_{\genvertex \in \stdcomplex_\genvertex(0)}{\sparens{\stdcochain^\genvertex(\genvertex)}}
    \]
\end{definition}
\begin{lemma}
    For every cochain $\stdcochain \in \cochainset{0}{\stdcomplex; \mathbbm{R}}$ it holds that:
    \[
        \lapwalkoperator[0]^{+}\stdcochain (\stdvertex) = \stdcochain(\stdvertex)-\lapwalkoperator[\stdvertex][0]^{-}\stdcochain^{\stdvertex}
    \]
\end{lemma}
\begin{proof}
    Notice that:
    \begin{align*}
        \lapwalkoperator[0]^{+}\stdcochain (\stdvertex) 
        & = \parens{I-A}\stdcochain(\stdvertex) 
        = \stdcochain(\stdvertex) - A\stdcochain(\stdvertex)
        = \stdcochain(\stdvertex) - \sum_{\genvertex \in \stdcomplex(0)}{\sparens{A}_{\stdvertex, \genvertex}\stdface(\genvertex)} \\
        & = \stdcochain(\stdvertex) - \sum_{\substack{\genvertex \in \stdcomplex(0)\\ \stdface\genface \in \stdcomplex(1)}}{\sparens{A}_{\stdvertex, \genvertex}\stdface(\genvertex)}
        = \stdcochain(\stdvertex) - \sum_{\genvertex \in \stdcomplex_\stdvertex(0)}{\weight[\stdvertex][\genvertex]\stdface^\stdvertex(\genvertex)}
        = \stdcochain(\stdvertex) - \lapwalkoperator[\stdvertex][0]^{-}\stdcochain^{\stdvertex}
    \end{align*}
\end{proof}
We are now ready to prove the trickling down theorem.
We will, however, present the proof to the Laplacian instead of the actual walk operator.
Theorem~\ref{thm:trickling-down-theorem} can be deduced using standard connection between the eigenvalue of an operator and its laplacian.
\begin{theorem}
    Let $X$ be a simplicial complex whose $1$-skeleton is connected and in which, for every $\stdvertex \in \stdcomplex(0)$ it holds that $\lambda_2(\lapwalkoperator[\stdvertex,0]) \ge \lambda$ then $\lambda_2(\lapwalkoperator[0]) \ge 2-\frac{1}{\lambda}$
\end{theorem}
\begin{proof}
    Let $\mu$ be a non trivial eigenvalue of $\lapwalkoperator[0]$ with the eigenfunction $\stdcochain$, i.e.:
    \[
        \lapwalkoperator[0] \stdcochain = \mu \stdcochain
    \]
    Note that due to Corollary~\ref{cor:averaging-over-co-boundary} it holds that:
    \begin{align}\label{eq:trickle-down:avg-over-restrictions}
        \mu \norm{\stdcochain}^2 & =
        \mu \innerprod{\stdcochain, \stdcochain} =
        \innerprod{\mu \stdcochain, \stdcochain} =
        \innerprod{\lapwalkoperator[0]^{+} \stdcochain, \stdcochain} =
        \ex_{\stdvertex \in \stdcomplex(0)}{\sparens{\innerprod{\lapwalkoperator[0,\stdvertex]^{+} \stdcochain^\stdvertex, \gencochain^\stdvertex}_{\stdvertex}}}
    \end{align}
    For every $\stdvertex \in \stdcomplex(0)$ let $\stdcochain^{\stdvertex \parallel} = \innerprod{\stdcochain^\stdvertex, \mathbbm{1}^v}_{\stdcomplex_{\stdvertex}}\mathbbm{1}^v$ be the projection of $\stdcochain^{\stdvertex}$ to constant on $\stdcomplex_{\stdvertex}$ and $\stdcochain^{\stdvertex \bot} = \stdcochain^\stdvertex - \stdcochain^{\stdvertex \parallel}$ and note that it is orthogonal to $\stdcochain^{\stdvertex \parallel}$.
    Note that for every $\stdvertex$ it holds that $\stdcochain^{\stdvertex\parallel}$ is constant over $\stdcomplex_{\stdvertex}$ and therefore:
    \[
        \lapwalkoperator[\stdvertex][0]^{+} \stdcochain^{\stdvertex\parallel}= (I-A)\stdcochain^{\stdvertex\parallel}=\stdcochain^{\stdvertex\parallel} - A \stdcochain^{\stdvertex\parallel} = \stdcochain^{\stdvertex\parallel} - \stdcochain^{\stdvertex\parallel} = 0
    \]
    Where $A\stdcochain^{\stdvertex\parallel}=\stdcochain^{\stdvertex\parallel}$ due to $A$ being an averaging operator and $\stdcochain^{\stdvertex\parallel}$ being constant.
    We denote by $\set{\gencochain_i}_{i \in I}$ the eigenfunction basis of $\lapwalkoperator[\stdvertex]^{+}$ when excluding the constant functions over $\stdcomplex_{\stdvertex}$ (i.e.\ the eigenfunctions whose eigenvalue is $0$).
    We then use the previous fact to conclude that:
    \begin{align*}
        \innerprod{\lapwalkoperator[\stdvertex][0]^{+}\stdcochain^\stdvertex, \stdcochain^\stdvertex}
        & = \innerprod{\lapwalkoperator[\stdvertex][0]^{+}\parens{\stdcochain^{\stdvertex\bot}+\stdcochain^{\stdvertex\parallel}}, \stdcochain^\stdvertex}
        = \innerprod{\lapwalkoperator[\stdvertex][0]^{+}\stdcochain^{\stdvertex\bot}, \stdcochain^\stdvertex}
        = \innerprod{\stdcochain^{\stdvertex\bot}, \lapwalkoperator[\stdvertex][0]^{+}\stdcochain^\stdvertex}\\
        & = \innerprod{\stdcochain^{\stdvertex\bot}, \lapwalkoperator[\stdvertex][0]^{+}\stdcochain^{\stdvertex\bot}}
        = \innerprod{\lapwalkoperator[\stdvertex][0]^{+}\stdcochain^{\stdvertex\bot}, \stdcochain^{\stdvertex\bot}}
        \innerprod{\lapwalkoperator[\stdvertex]^{+} \sum_{i \in I}{\alpha_i \gencochain_i}, \stdcochain^{\stdvertex \bot}} = \\
        & = \sum_{i \in I}{\alpha_i \innerprod{\lapwalkoperator[\stdvertex]^{+} \gencochain_i, \stdcochain^{\stdvertex \bot}}} =
        \sum_{i \in I}{\alpha_i \innerprod{\lambda_i \gencochain_i, \stdcochain^{\stdvertex \bot}}} =
        \sum_{i \in I}{\alpha_i \lambda_i \innerprod{\gencochain_i, \stdcochain^{\stdvertex \bot}}} \ge \\
        & \ge \lambda \sum_{i \in I}{\alpha_i \innerprod{\gencochain_i, \stdcochain^{\stdvertex \bot}}} =
        \lambda \innerprod{\sum_{i \in I}{\alpha_i \gencochain_i}, \stdcochain^{\stdvertex \bot}} =
        \lambda \innerprod{\stdcochain^{\stdvertex \bot}, \stdcochain^{\stdvertex \bot}} =
        \lambda \norm{\stdcochain^{\stdvertex \bot}}^2
    \end{align*}
   We combine this with (\ref{eq:trickle-down:avg-over-restrictions}) to conclude that:
    \[
        \mu \norm{\stdcochain}^2 =  
        \ex_{\stdvertex \in \stdcomplex(0)}{\sparens{\innerprod{\lapwalkoperator[0,\stdvertex]^{+} \stdcochain^\stdvertex, \gencochain^\stdvertex}_{\stdvertex}}} 
        \ge \lambda \ex_{\stdvertex \in \stdcomplex(0)}{\sparens{\norm{\stdcochain^{\stdvertex\bot}}^2}}
    \]
    We will move on to calculate $\ex_{\stdvertex \in \stdcomplex(0)}{\sparens{\norm{\stdcochain^{\stdvertex\bot}}}}$:
    Note that:
    \begin{align}\label{eq:trickle-down:calculate-perpadicular}
    \begin{split}
        \norm{\stdcochain^{\stdvertex}}^2 
        & = \innerprod{\stdcochain^{\stdvertex},\stdcochain^{\stdvertex}} 
        =\innerprod{\stdcochain^{\stdvertex \parallel} + \stdcochain^{\stdvertex \bot}, \stdcochain^{\stdvertex \parallel} + \stdcochain^{\stdvertex \bot}} \\
        & =\innerprod{\stdcochain^{\stdvertex \parallel},\stdcochain^{\stdvertex \parallel}} + \innerprod{\stdcochain^{\stdvertex \bot}, \stdcochain^{\stdvertex \bot}} 
        =\norm{\stdcochain^{\stdvertex \parallel}}^2 + \norm{\stdcochain^{\stdvertex \bot}}^2
    \end{split}
    \end{align}
    In addition, due to Lemma~\ref{lem:average-over-links}:
    \begin{equation}\label{eq:trickle-down:average-over-links}
        \norm{\stdcochain}^2 = \ex_{\stdvertex \in \stdcomplex(0)}{\sparens{\norm{\stdcochain^{\stdvertex}}^2}}
    \end{equation}
    We note that:
    \[
        \lapwalkoperator[\stdvertex]^{-}\stdcochain^\stdvertex(\genvertex) =
        \ex_{\genvertex' \in \stdcomplex(0)}{\sparens{\stdcochain^\stdvertex(\genvertex')}}\mathbbm{1}^{\stdvertex}(\genvertex) =
        \parens{\sum_{\genvertex' \in \stdcomplex(0)}{\weight[\stdvertex][\genvertex']\stdcochain^\stdvertex(\genvertex')\mathbbm{1}^{\stdvertex}(\genvertex')}}\mathbbm{1}^{\stdvertex}(\genvertex) =
        \innerprod{\stdcochain^\stdvertex, \mathbbm{1}^\stdvertex}\mathbbm{1}^\stdvertex
    \]
    And conclude that:
    \[
        \mu \stdcochain(\stdvertex) = \stdcochain(\stdvertex) - \stdcochain^{\stdvertex\parallel} \Rightarrow \stdcochain^{\stdvertex\parallel} = \parens{1-\mu}\stdcochain(\stdvertex)
    \]
    We can now calculate $\ex_{\stdvertex \in \stdcomplex(0)}{\sparens{\norm{\stdcochain^{\stdvertex \parallel}}^2}}$:
    \begin{align}\label{eq:trickle-down:calculating-parallel}
        \ex_{\stdvertex \in \stdcomplex(0)}{\sparens{\norm{\stdcochain^{\stdvertex \parallel}}^2}} =
        (1-\mu)^2 \ex_{\stdvertex \in \stdcomplex(0)}{\sparens{\norm{\stdcochain}^2}}  =
        (1-\mu)^2 \norm{\stdcochain}^2
    \end{align}
    Combining~\ref{eq:trickle-down:calculate-perpadicular}, \ref{eq:trickle-down:average-over-links} and \ref{eq:trickle-down:calculating-parallel} we get that:
    \begin{align*}
        \ex_{\stdvertex \in \stdcomplex(0)}{\sparens{\norm{\stdcochain^{\stdvertex\bot}}}} & =
        \ex_{\stdvertex \in \stdcomplex(0)}{\sparens{\norm{\stdcochain^{\stdvertex}}^2}} - \ex_{\stdvertex \in \stdcomplex(0)}{\sparens{\norm{\stdcochain^{\stdvertex \parallel}}^2}} =
        \norm{\stdcochain}^2 - (1-\mu)^2 \norm{\stdcochain}^2 = \\
        & = \mu(2-\mu)\norm{\stdcochain}^2
    \end{align*}
    Using this we conclude that:
    \[
        \mu \norm{\stdcochain}^2 \ge \lambda\mu(2-\mu)\norm{\stdcochain}^2
    \]
    And thus:
    \[
        1 \ge \lambda(2-\mu) \Rightarrow \frac{1}{\lambda} \ge 2-\mu \Rightarrow \mu \ge 2-\frac{1}{\lambda}
    \]
    As required by the Lemma.
\end{proof}
Applying the trickling theorem repeatedly yields the following local to global result:
\begin{corollary}
    Let $\stdcomplex$ be a $d$-dimensional pure simplicial complex.
    If there exists $\lambda \in \left(0,1\right]$ such that:
    \begin{itemize}
        \item For every $\genface \in \stdcomplex$ such that $\dim\parens{\genface} \le d-2$ it holds that $\stdcomplex_\genface^{(1)}$ is connected.
        \item For every $\genface \in \stdcomplex$ such that $\dim\parens{\genface} = d-2$ it holds that $\lambda_2\parens{\stdcomplex_\genface^{(1)}} \le \frac{\lambda}{1+\parens{d-1}\lambda}$.
    \end{itemize}
    Then $\stdcomplex$ is a $\lambda$-local spectral expander.
\end{corollary}
\section{Local to Global Topological Expansion of High Dimensional Expanders}\label{Sec:local-to-global-topological-expansion}
In the previous section we have seen local to global spectral expansion over $\mathbb{R}$.
As we have explained, that result is based on the Garland philosophy using the fact that over the reals we have inner products and self adjoint operators, which are useful in deriving the local to global theorem. 

In order to prove local to global expansion in the topological sense, we need to show a local to global statement that occurs over finite fields (in particular, over $\mathbb{F}_2$) in the case of high dimensional expanders. Thus, we have to deviate dramatically from the Garland paradigm that uses self adjoint operators and inner products as they do not exist over $\mathbb{F}_2$.

The local to global method we develop here (that is used to prove the local to global expansion over $\mathbb{F}_2$) uses the following idea: 
It shows how to derive expansion of small sets on a complex with topologically expanding links using a newly introduced notion of \emph{local minimality}.
It then shows that expansion of large sets can be inferred from expansion of small sets in the case of a complex with high enough dimensions. 

The following theorem is a central local-to-global theorem in topological high dimensional expansion.
It essentially says that global topological expansion denoted as cosystolic expansion can be obtained from local topological expansion known as coboundary expansion.

We see here again the philosophy that in high dimensional expansion we have a local to global deduction with some loss.
I.e, from local coboundary expansion we deduce a global cosystolic expansion which is a weaker topological notion of expansion.
Recall that, in the local to global spectral expansion, we have a loss in the spectral expansion the more we moved down in the trickling procedure. 

The local to global topological expansion theorem first appeared in~\cite{kaufman2016isoperimetric} albeit only for dimension $2$. 
It was then extended to any dimension in~\cite{EvraK16}, and recently was extended further by~\cite{KaufmanM21} to show that cosystolic expansion could be defined with regard to not only binary cochains, but rather to cochains that get values in any group.
The work of~\cite{KaufmanM21} shows that, even under this more generalized setting, global cosystolic expansion could be deduced from coboundary expansion in links. 

\begin{theorem}[Local-to-global cosystolic expansion~\cite{kaufman2016isoperimetric,EvraK16,KaufmanM21}] For any $d, q \in N$ and $0 < \beta < 1$,  there exist $0 < \lambda, \eta < 1$ such that the following holds: Let $X$ be a $d$-dimensional $q$-bounded degree simplicial complex satisfying the following local conditions:
\begin{itemize}
    \item Spectral expansion in links: $X$ is one sided-$\lambda$-local spectral expander.
    \item Topological expansion in links: $X$'s links are $\beta$-coboundary expanders.
\end{itemize}
Then the $(d-1)$-skeleton of $X$ is an $(\epsilon, \mu)$-global cosystolic expander, where $$\epsilon = \mbox{min}\set{\eta^{2^d-1},\frac{1}{qd^{\frac{d}{2}}}} \mbox{ and } \mu=\eta^{2^d - 1}$$
\end{theorem}

As previously mentioned, an important concept that was introduced in~\cite{kaufman2016isoperimetric} and is a major tool in the analysis of the above theorem is a notion called \emph{local minimality}. 

\begin{definition}[Local minimality~\cite{kaufman2016isoperimetric}]
Given a weighted simplicial complex $(X,w)$, a $k$-cochain $f \in C^k(X)$ is (globally) minimal if 
$||f|| = \mbox{min}_{b \in B^k(X)}\{||f+b||\}$.
$f$ is called \emph{locally minimal} if $f_{\sigma}$ is minimal in $X_{\sigma}$ for every $\sigma \in X(i)$, $0 \leq  i < k$.
\end{definition}

Note that minimality implies local minimality but not vice versa.
The main idea of~\cite{kaufman2016isoperimetric} and its followup works was to use the notion of local minimality for showing that small sets topologically expand. Namely, for showing that a simplicial complex with local links that are both spectrally expanding and topologically expanding is, in fact, a \emph{small set coboundary expander}, which is defined as follows.  

\begin{definition}[Small set coboundary expander: a complex in which small sets topologically expand~\cite{kaufman2016isoperimetric}]
A $d$-dimensional weighted simplicial complex $(X,w)$ is called \emph{$(\epsilon,\mu)$-small set coboundary expander} for some constants $0 < \epsilon,\mu \leq  1$, if for every $k$-cochain $f \in C^k(X)$, $k<d$
with $||f||<\mu$ it holds that $||\delta f|| > \epsilon ||f||$.
\end{definition}

Thus, ~\cite{kaufman2016isoperimetric,EvraK16,KaufmanM21} have used the notion of local minimality to show that a complex with expanding links is a small-set coboundary expander (i.e., it topologically expands small sets). Specifically, the following was shown.

\begin{theorem} [Small set coboundary expansion from expanding links~\cite{kaufman2016isoperimetric,EvraK16}]
A $q$-bounded degree $d$ dimensional complex whose links are sufficiently strong $\lambda$-local spectral expanders and $\beta$-coboundary expanders is a $d$ dimensional $(\epsilon,\mu)$-small set coboundary expander where $(\epsilon,\mu)$ depends on $\lambda,\beta, q$.
\end{theorem}

The next major idea of~\cite{kaufman2016isoperimetric,EvraK16,KaufmanM21} is 
that there is a way to deduce expansion of all sets (in particular, large sets) from expansion of small sets, assuming the complex is of high enough dimension. Specifically, the following was shown: 

\begin{theorem} [Cosystolic expansion from small set coboundary expansion \cite{kaufman2016isoperimetric,EvraK16}]
If $(X,w)$ is a $d$-dimensional weighted simplicial complex of bounded degree that is a $(\epsilon,\mu)$-small set coboundary expander then its $(d-1)$-skeleton is a $(\epsilon,\mu)$-cosystolic expander.
\end{theorem}

Using all the above and the known existence of bounded degree complexes with expanding links the following was deduced.

\begin{corollary} [There are bounded degree cosystolic expanders of every dimension \cite{kaufman2016isoperimetric,EvraK16}]
For every $d \ge 1$, the $d$-skeleton of the $d+1$ dimensional Ramanaujan complex~\cite{lubotzky2005explicit} is a bounded degree cosystolic expander of dimension $d$. 
\end{corollary}

This local to global paradigm that we have introduced was proven to be very useful recently in order to deduce cosystolic expansion in covers of high dimensional expanders~\cite{GotlibKaufman22, FirstKaufman22} as the links structure of the cover of a simplicial complex is the same as in the base complex.  

As we next discuss, the topological high dimensional expansion is tightly related  to local testability of codes.
The local to global proof of the cosystolic expansion that we have discussed here can be seen, via this prism, as a method to get global local-testability of codes from local local-testability.
This philosophy was then implemented in various subsequent works.

\section{High Dimensional Expansion and Local Testability of Codes}\label{Sec:LTCs}

One of the major motivations for studying high dimensional expanders within theoretical computer science (TCS) was the discovery of their strong relation to a central notion within TCS known as a locally testable code. Locally testable codes were extensively studied, however, their study was mostly ah hoc and there was not known a mathematical phenomenon that implies local testability of codes. 

The discovery that topological high dimensional expansion is equivalent to local testability of some particular codes suggested that local testability of codes is implied by high dimensional expansion. In order to indeed confirm this ideology, there was a need to show that (1) known locally testable codes can be explained via the high dimensional expansion prism, as well as (2) a method to get new LTCs from topological high dimensional expanders. Both of these goals were materialized using the local-to-global effect existing in high dimensional expanders, as we following discuss. However, it turned out the LTCs emerging from high dimensional expanders are not of high enough rate.

Recent works have shown how to overcome the rate issue existing in LTCs emerging from high dimensional expanders, and get high rate LTCs from product of two (one dimensional) expander graphs. These works have adjusted the local-to-global machinery developed in the realm of high dimensional expanders, to work for product of expanders, thus overcoming the rate barrier existing in using genuine high dimensional expanders to get LTCs. We, nevertheless conjecture that high rate LTCs with stronger guaranties should emerge from genuine high dimensional expanders.

We start by defining locally testable codes.

\begin{definition}[Locally Testable Code] A locally testable code (LTC) is an error correcting code admitting a randomized algorithm---called a \emph{tester}---which, given access to a word, can decide with high probability whether it is close to a codeword or not by querying just a few (i.e.\ $O(1)$) of its letters. A tester is called an $\epsilon$-tester for some $\epsilon >0$ if it accepts all codewords, and the probability of rejecting a word outside the code is $\epsilon$ proportional to its Hamming distance from the code. An LTC with an $\epsilon$-tester is called an $\epsilon$-locally testable code. 
\end{definition}

\paragraph{Topological high dimensional expansion is a form of local testability of codes.} The first discovery of the connection between high dimensional expanders and locally testable codes was made by~\cite{KaufmanLubotzky14} where the authors have shown that a coboundary expansion is, in fact, equivalent to the local testability of the coboundary code.

\begin{theorem}[Coboundary expansion is equivalent to local testability of the coboundary code~\cite{KaufmanLubotzky14}] 
A $d$-dimensional complex $X$ is $\epsilon$-coboundary expander iff the linear code $B^i(X)$, $i<d$ is $\epsilon$-locally testable by the $(i+1)$-cocycle test. This test, is given access to $f\in C^i(X)$, choose a face $\sigma \in X(i+1)$ uniformly at random and accept $f$ if $\delta_i f (\sigma) = \sum_{\tau \in \binom{\sigma}{i+1}}{f(\tau)}$.
\end{theorem}

\paragraph{Global local-testability from local local-testability via high dimensional expanders.}\label{ch:high-dimensional-expansion-and-testability}

The works~\cite{kaufman2016isoperimetric,EvraK16,KaufmanM21} have, in fact, shown that a global code defined over a high dimensional expander, by requiring its projection to every link to belong to a small (local) locally-testable code, is a global locally-testable code. Namely, the local testability of the global code stems from the local testability of the small codes that compose it. In these works the global locally-testable code is $H^i(X)$ where $X$ is a cosystolic expander, and the local locally-testable codes are $B^i(X_{\sigma})$ for $\sigma \in X$.

\begin{theorem} [Cosystolic expansion implies locally testable codes with linear distance~\cite{kaufman2016isoperimetric,EvraK16,KaufmanM21}] 
Let $X$ be a $d$-dimensional complex which is $(\epsilon,\mu)$-cosystolic expander then for every $i\leq d-2$, $H^i(X)$ is a $\epsilon$-locally testable code, whose normalized distance is at least $\mu$. For $i=d-1$ the code $H^i(X)$ is of distance at least $\mu$. 
\end{theorem}

Using the last theorem one can get locally testable codes associated, for example, with $H^i(X)$ of an $(i+2)$-dimensional cosystolic expander $X$, $i\geq 0$; However the rate of such codes tend to be small.

Thus, the lesson is that cosystolic expanders imply locally testable codes via the local local-testability implies global local-testability paradigm, but such codes tend to be of small rate.
Nevertheless, the philosophy of global local-testability from local one is very strong and it has recent important implications to locally testable codes as we now discuss. 

\paragraph{Explaining and improving testability of known LTCs via high dimensional expanders.}
This idea of global local-testability from local ones was used recently by~\cite{kaufman2021high}, see also~\cite{dikstein2020locally}, to re-prove the local testability of single orbit affine invariant codes~\cite{DBLP:conf/stoc/KaufmanS08} via the high dimensional expansion paradigm. The new analysis has also provided tighter bounds than were previously obtained.

\paragraph{High rate LTCs.} Another important implication of the local to global testability occurs in the recent works of~\cite{dinur2021, PK21} that have constructed high rate (good) locally testable codes from twisted product of expander graphs.
Using product of expander graphs, they were able to adopt the philosophy of local to global testability to construct new LTCs of high rate. The importance of turning to product of expander graphs was to overcome the small rate barrier that existed in attempts to implement the local to global testability from two dimensional genuine high dimensional expanders.

Given the recent discovery of good LTCs constructed from two dimensional objects obtained from a product of two graphs, it is natural to ask whether good LTCs could be constructed from genuine high dimensional expanders, and what advantages such constructions might have over current ones.  

\paragraph{On $2$-LTCs from genuine high dimensional expanders} A recent work by~\cite{FirstKaufman22} suggests a framework to get good $2$-queries LTCs from genuine high dimensional expanders by introducing expanding high dimensional sheaves.
The currently known good LTCs are of very high locality.
However, for hardness of approximation, it is desirable to get 2-queries LTCs.
We conjecture that LTCs emerging from genuine high dimensional expanders should be much stronger in various respects than the ones constructed from products of one-dimensional graphs, however these seem much harder to construct.

\section{High Dimensional Expansion and Quantum LDPC Codes}\label{Sec:QLDPC-codes}

Good classical LDPC codes were known since the works of Gallager in the 60's \cite{GAL}. However, their quantum analogous seemed elusive until recently. The problem being that random LDPC classical codes can be shown to be good with high probability, while a quantum LDPC code is composed of two co-dual classical LDPC codes (see following definition). Such a pair of co-dual classical LDPC codes that are both good can not be obtained by random means, hence there was no natural source for obtaining such good codes. 

\begin{definition} [Quantum LDPC code (CSS code) \cite{CSS}]
A \emph{quantum LDPC CSS code} is a quintet
$C=(C_X,C_Z,\mathbb{F}^n,\Phi_X,\Phi_Z)$ such that the $X$-code $C_X$
and the $Z$-code $C_Z$ are subspaces of $\mathbb{F}^n$,   $\Phi_X$ is a set of vectors generating $C_X^\perp$,
$\Phi_Z$ is a set of vectors generating $C_Z^\perp$, and $C_X^\perp\subseteq C_Z$ (equivalently,
$C_Z^\perp\subseteq C_X$). The code is LDPC if each vector in  $\Phi_X,\Phi_Z$ has constant (independent of $n$) support. 
The rate of $C$ is $\dim C_X-\dim C_Z^\perp=\dim C_Z-\dim C_X^\perp$
and its distance is $\min\set{d_X,d_{Z}}$, where $d_X=\min\set{\abs{w}_{\Ham}\where w\in C_X-C_Z^\perp}$
and $d_{Z}=\min \{\|w\|_{\Ham}\where w\in C_Z-C_X^\perp\}$;
we call $d_X$ and $d_Z$ the $X$- and $Z$-distance, respectively.
The relative distance and rate of $C$ are its distance and rate divided by $n$, respectively.
\end{definition}

Up until recently most quantum LDPC codes were obtained from smooth topological objects with very simple local structure (surfaces, manifolds), see e.g. the Toric codes \cite{Toric}. This is since in such cases there is a natural way to get a pair of such codes that the dual code is isomorphic to the primal code (thus, there is a need only to design one code). The best distance achieved by such codes has exceeded slightly $\sqrt{n}$ by the notable work of Freedman et. al. \cite{FML} from about 20 years ago. Going beyond the $\sqrt{n}$ distance barrier of Freedman et. al. with any rate seemed beyond reach for many years. Very recently, \cite{EvraKZ20,KaufmanT21quantum} managed to get quantum LDPC codes improving the distance record of Freedman et. al. to $\sqrt{n} \log^k n$ for any $k$, using a \emph{novel approach based on high dimensional expanders}. Furthermore the codes of \cite{EvraKZ20} have fast decoding algorithms. 

\begin{theorem} [Decodeable quantum LDPC codes beyond the $\sqrt{n}$ distance barrier \cite{EvraKZ20}]
There exist quantum LDPC codes of distance $O(\sqrt{n} \log n)$ and rate $O(\sqrt{n}/ \log n)$ that are efficiently decodeable.
\end{theorem}

In fact, the work of \cite{EvraKZ20} has used the following realization:

\paragraph{Quantum LDPC codes are well as classical LTCs are born \emph{together} from a single object - a topological high dimensional expander.}  

\begin{theorem}[Cosystolic expanders imply simultaneously quantum codes with linear $X$-distance and a locally testable code with a linear distance \cite{kaufman2016isoperimetric,EvraK16,KaufmanM21}] Two dimensional cosystolic expander $X$ implies a quantum LDPC code whose $X$-code corresponds to $H^1(X)$, its $Z$-code corresponds to $H_1(X)$; thus the $X$-code has linear distance. The rate of the code is $dim H^i(X)$. Furthermore, $H^0(X)$ is a locally testable code with linear distance, whose rate is $dim H^0(X)$.
\end{theorem}

Thus, two dimensional cosystolic expander gives both a locally testable code corresponding to $H^0(X)$ and a Quantum LDPC code corresponding to $H^1(X)$ and $H_1(X)$.

Applying this wisdom to cosystolic expanders arising from Ramanujan complexes \cite{lubotzky2005explicit} implies a quantum code with linear $X$-distance, logarithmic $Z$-distance and non-zero rate. One can even use higher dimensional cosystolic expanders to get that the $X$-code has both good distance and it is a locally testable code!

The idea of \cite{EvraKZ20}, following Hastings, was to apply the following balancing procedure that balances the $X$-distance and the $Z$-distance to get a quantum code whose distance is a geometric average of the two. 

\begin{theorem} [Weight balancing procedure for quantum LDPC codes \cite{EvraKZ20}]
There exists a way to product a two-dimensional cosystolic expander $X$ with a one-dimensional expander graph to get a quantum code, whose distance is the geometric average of the distances of $H^1(X)$ and $H_1(X)$ and whose rate is roughly $\sqrt{n}$. 
\end{theorem}

Following the works \cite{EvraKZ20,KaufmanT21quantum} there were flurry of improvements starting with \cite{HHO} (see also \cite{NE21}) that have replaced the product of a two-dimensional simplicial complex with a graph with a twisted product of two graphs. The resulted object has been a two dimensional expanding cubical complex. These advances culminated in the work of \cite{PK21} that found a way to use methods inspired by the proof of local to global cosystolic expansion of \cite{kaufman2016isoperimetric} and, in particular the local minimality concept, in order to prove linear distance for these codes, thus obtaining good quantum LDPC codes. Recall that \cite{kaufman2016isoperimetric} managed to prove linear distance only to the $X$-code. By turning to cubical complexes \cite{PK21} managed to derive the linear distance both to the $X$-code and to the $Z$-code. The dimension of the code arose simply from counting degrees of freedom.

Another major question in the field of quantum codes is whether there are good quantum LTCs. We conjecture that high dimensional expansion will be a key towards progress on finding such codes. 

\section{Gromov Topological Overlapping Problem via Topological High Dimensional Expanders}\label{sec:TOP}

More than a decade ago Misha Gromov, one of the greatest mathematicians of our era, defined the notion of \emph{topological overlapping property} of a simplicial complex and posed the following question: Are there bounded degree simplicial complexes with the topological overlapping property? In the following we discuss how the topological high dimensional expansion perspective, and in particular its connection to locally testable codes, has been recently used to provide a positive answer to Gromov's question by combining the works \cite{kaufman2016isoperimetric,EvraK16,DotterrerKW16}.

We start by introducing the topological overlapping property.

\begin{definition} [Topological overlapping property \cite{Gro10}]
A $d$-dimensional complex $X$ is said to have the \emph{topological overlapping property}  if for any continuous map $F : X \rightarrow R^d$ , there exists a point $p \in  \R^d$, such that $F^{-1}(p)$ is covered by $\epsilon > 0$ fraction of the $d$-faces of $X$.
\end{definition} 

In order to study the question on the existence of bounded degree complexes with topological overlapping property Gromov has introduced the notion of coboundary expansion of simplicial complexes (that was introduced independently by Linial and Mehsulam \cite{LM06}). Gromov has shown that coboundary expansion implies the topological overlapping property; Doing that he, in fact, showed that there are unbounded degree complexes with the topological overlapping property. However, the question on the existence of bounded degree complexes with this property remained unsolved at that point. 

By the connection Gromov made between coboundary expansion (which is a notion of local testability of code as we have explained! ) to the topological overlapping question it became evident that in order to answer positively Gromov's question one should find bounded degree coboundary expanders.

The natural candidates for such bounded degree coboundary expanders were the famous bounded degree Ramanujan complexes \cite{lubotzky2005explicit}. Alas, these complexes are known not to be coboundary expanders. 
However, the local to global cosystolic result of \cite{kaufman2016isoperimetric, EvraK16} could, in fact, be applied to these complexes to show that they poses the weaker property of cosystolic expanders, thus they come close to satisfy the condition implying the topological overlapping  property.

Thus, in order to give a positive resolution to Gromov's question one had to prove that small set coboundary expansion (which implies cosystolic expansion) is, in fact, sufficient for the topological overlapping property and, this was indeed proved by \cite{DotterrerKW16}.

Thus, combining the works of \cite{kaufman2016isoperimetric,EvraK16} and \cite{DotterrerKW16} led to a positive resolution of Gromov's question on the existence of bounded degree complexes with the topological overlapping property.

Another interesting aspect of this question is the following. Since we know that cosystolic expansion is a form of local testability of codes, we, in fact, see an evident connection between locally testable codes and the topological overlapping property. This connection might be proven useful in future application within theoretical computer science.

\section{Sampling Bases of a Matroid Using Local Spectral Expanders}\label{sec:sampling-matroids}
In this section we will present Anari, Liu, Gharan and Vinzant's~\cite{anari2019log} resolution of the Mihail-Vazirani conjecture \cite{mihail1989expansion} by rigorously showing how rapid convergence of the down-up walk over local spectral expanders~\cite{kaufman2018high} were used in their solution.
Before we can present the conjecture, consider the following definition:
\begin{definition}[Matroid]
    A matroid $M=(X,\mathcal{I})$ is a combinatorial structure consisting of a ground set $X$ of elements and a non empty collection $\mathcal{I}$ of independent subsets of $X$ satisfying:
    \begin{enumerate}
        \item \textbf{Hereditary Property:} If $T \in \mathcal{I}$ and $S \subseteq T$ then $S \in \mathcal{I}$.
        \item \textbf{Exchange Axiom:} If $T_1, T_2 \in \mathcal{I}$ and $\abs{T_2} < \abs{T_1} $ then there exists $i \in T_1 \setminus T_2$ such that $T_2 \cup \set{i} \in \mathcal{I}$.
    \end{enumerate}
    We call maximal independent sets in $M$ the \emph{bases} of $M$.
\end{definition}

Matroids are a natural combinatorial objects and can be thought of as generalizations of a basis of some vector field or sub-forests of a graph.
It is natural to try to sample a basis of matroid (i.e.\ sample a basis of some vector field or a spanning tree of a graph).
One natural process to try and sample a random basis of a matroid is the using the ``base exchange random walk'' which is a random walk over the bases of a matroid, defined as follows:
\begin{definition}[The base exchange random walk]
The bases exchange walk is a walk between bases of a matroid whose step is defined by the following process:\\
\begin{algorithm}[H]
    \caption{Step in the base exchange walk}
    \DontPrintSemicolon
    Pick a member of the base is chosen uniformly at random and is then removed creating the independent set $I$.\;
    Pick a new base that contains $I$.\;
    \Return{the new basis.}
\end{algorithm}
\end{definition}
Mihail and Vazirani conjectured the following:    
\begin{conjecture}[Mihail and Vazirani Conjecture \cite{mihail1989expansion}]
The bases exchange walk converges rapidly.
\end{conjecture}

We note that, due to the hereditary property, one can think of a simplicial complex that is comprised of the independent sets of a matroid.
Moreover, it is easy to see that the matroid exchange walk is the $d$-dimensional down-up walk!
We will therefore be interested in the expansion properties of that simplicial complex - as, if that complex is indeed a local spectral expander, then the base exchange walk converges rapidly and we could sample a basis of the matroid by picking some constant basis to start from (which corresponds to a face of maximal dimension) and then perform a few steps in the down-up walk.
Then we can conclude that we have arrived at a random basis due to the convergence of that walk.
We will show that matroids are the best possible local spectral expanders, $0$-local spectral expanders.
Note that the simplicial complex that is constructed by the independent sets of the matroid satisfies the following property:
\begin{definition}[Exchange property]
    We say that a simplicial complex satisfy the exchange property if the following holds:
    For every two faces $\stdface, \genface \in \stdcomplex$ such that $\abs{\stdface} < \abs{\genface}$ there exists $\stdvertex \in \genface \setminus \stdface$ such that $\stdface \cup \set{\stdvertex} \in \stdcomplex$.
\end{definition}
It is easy to see that if a simplicial complex satisfies the exchange property then so are all its links and skeletons.
We will therefore try to use the trickling down theorem in order to prove that matroids are indeed local spectral expanders.
We will start with proving that all the links of co-dimension $2$ are $0$-local-spectral-expanders:
\begin{lemma}\label{lem:exchange-division-lemma}
If $G$ satisfies the exchange property and $E \ne \emptyset$ then there exists a partition of $V$ into $V_1, V_2, V_3$ such that:
\begin{itemize}
    \item $V_1$ and $V_2$ are not empty \infullversion{.
    \item $V_1$ and $V_2$ are independent.
    }{and independent.}
    \item For all $i \ne j$: For every $u \in V_i$ and $v \in V_j$ it holds that $\set{u,v} \in E$.
\end{itemize}
\end{lemma}
\begin{proof}
    Let $\set{u,v} \in E$ and consider the following sets:
    \begin{align*}
        V_u = \set{w \suchthat \set{v,w} \in E, \set{u,w} \notin E}&, \quad V_v = \set{w \suchthat \set{u,w} \in E, \set{v,w} \notin E}\\
        V_{\set{u,v}} = & \set{w \suchthat \set{v,w} \in E, \set{u,w} \in E}
    \end{align*}

    We will show that $V_u$ is independent, the same face about $V_v$ is analogous:
    Assume that $V_u$ is not independent therefore there exists $w, w' \in V_u$ such that $\set{w,w'} \in E$.
    Therefore, due to the exchange property either $\set{v,w} \in E$ or $\set{v,w'} \in E$ which contradicts our choice of $w$.

    We will start by showing that any vertex in $V_u$ is connected to every vertex in $V_v$:
    Let $u' \in V_u$ and $v' \in V_v$.
    $u' \in V_u$ therefore $\set{u',u} \in E$.
    Applying the exchange property to $\set{u',u}$ and $w'$ yields that either $\set{v',v} \in E$ or $\set{u',v'} \in E$.
    Therefore due to the definition of $V_u$ it holds that $\set{v',v} \notin E$ and thus $\set{u',v'} \in E$.

    We follow by showing that any vertex is $V_u$ is connected to any vertex in $V_{\set{u,v}}$ (the case of $V_v$ is analogous):
    Let $u' \in V_u$ and $w \in V_{\set{u,v}}$.
    $w \in V_{\set{u,v}}$ therefore $\set{w, v} \in E$.
    Applying the exchange property to $\set{w,v}$ and $u'$ yields that either $\set{w,u'} \in E$ or $\set{v,u'} \in E$.
    Therefore due to the definition of $V_u$ it holds that $\set{u',v} \notin E$ and thus $\set{w,u'} \in E$.

    We finish the proof by setting:
    \[
        V_1 = V_u \cup \set{v}, \quad V_2 = V_v \cup \set{u}, \quad V_3 = V_{\set{u,v}}
    \]
    And noting that the properties listed in the lemma hold for these sets.
\end{proof}
We will use this Lemma to prove the following:
\begin{lemma}[Graphs with the exchange property are $0$-spectral-expanders]\label{lem:expansion-of-d-2-links}
Let $G$ be a pure graph that satisfies the exchange property then $G$ is a $0$-spectral-expander.
\end{lemma}
\begin{proof}
    We will show that $G$ is a complete partite graph and therefore a $0$-spectral-expander. 
    Using Cauchy's Interlacing Theorem on the non-normalized adjacency matrix and its complement yields that the non-normalized matrix's second eigenvalue is bounded from above by $0$. The normalization can then be performed by noting that it is equivalent to multiplication from the right by a positive semidefinite matrix. Then one can use Cauchy's Interlacing Theorem again in order to show that the resulting matrix has at most one positive eigenvalue.
    
    We will do so by constructing said partition using recursive application of Lemma~\ref{lem:exchange-division-lemma}.
    Start by setting $\tilde{V} = V$ and set the partition to be $\mathcal{V} = \emptyset$.
    While there are edges in the subgraph induced by $\tilde{V}$ has edges do the following: apply Lemma~\ref{lem:exchange-division-lemma} to $\tilde{V}$, add $V_1$ and $V_2$ to $\mathcal{V}$ and set $\tilde{V} = V_3$.
    Note that every set in $\mathcal{V}$ is independent due to Lemma~\ref{lem:exchange-division-lemma}.
    Also note that if $U_1, U_2 \in \mathcal{V}$ are two different sets then every vertex in $U_1$ is connected to every vertex in $U_2$.
    Therefore $G$ is a complete partite graph.
\end{proof}
Now all we have to check is that all of the links' $1$-skeletons are connected, as we will in the following Lemma:
\begin{lemma}[Connectivity of Graphs that Satisfy the Exchange Property]\label{lem:connectivity-of-exhnage-graphs}
Let $X$ be a pure simplicial that satisfies the exchange property then $X^{(1)}$ is connected.
\end{lemma}
\begin{proof}
    Let $\stdvertex,\genvertex \in X^{(1)}(0)$.
    $X^{(1)}$ is pure therefore there exists $\tilde{\stdvertex}$ such that $\set{\stdvertex, \tilde{\stdvertex}} \in E$.
    $X^{(1)}$ is a skeleton of a complex that satisfies the exchange property and therefore it satisfies the exchange property as well, meaning that wither $\set{\stdvertex, \genvertex} \in E$ or $\set{\tilde{\stdvertex}, \genvertex} \in E$.
    Therefore $\stdvertex$ and $\genvertex$ either are connected directly or through $\tilde{\stdvertex}$ and thus $X^{(1)}$ is connected.
\end{proof}
Therefore we can conclude that:
\begin{theorem}
    If $M=(X,\mathcal{I})$ is a matroid then $\mathcal{I}$ is a $0$-local spectral expander.
\end{theorem}
\begin{proof}
    Combining Lemma~\ref{lem:expansion-of-d-2-links}, Lemma~\ref{lem:connectivity-of-exhnage-graphs} and the tricking down theorem proves this theorem.
\end{proof}
Therefore high dimensional expanders can be used to resolve the Mihail-Vazirani conjecture simply by noting the following:
The base exchange walk is the down-up walk on the maximal dimension of a simplicial complex that exhibits the exchange property.
Any simplicial complex that exhibits the exchange property is a $0$-local spectral expander and therefore, as we presented in Theorem~\ref{thm:ko-random-walks} the second largest eigenvalue of the down-up walk is smaller than $1-\frac{1}{k+1}$ which is bounded away from $1$.
Therefore the base exchange walk converges rapidly.

We end this chapter by noting that there are strong connections between counting and sampling for self reducible problems~\cite{jerrum1986random} and therefore being able to sample a random basis for the matroid also proves the existence of a randomized algorithm that estimates the number of bases the matroid has.

\section{Additional topics that are not covered in this note}\label{Sec:additional}
Before ending this note we mention some topics related to high dimensional expansion that were \emph{not} covered in this note. These topics include: discussion of explicit constructions of high dimensional expanders. In particular bounded degree constructions. Here we note that currently there are no combinatorial constructions of strong local spectral expanders, but only algebraic ones. 
High dimensional expanders beyond simplicial complexes, for example the Grassmanian complex and its properties. Concentration of measure via high dimensional expanders, Fourier analysis and hypercontractivity on high dimensional expanders; Agreement expanders and low degree testing via high dimensional expanders, Super fast mixing of Markov chains and Glauber dynamics via strong high dimensional expanders. Unique Games and high dimensional expanders.

\begin{funding}
Research supported by ERC and BSF.
\end{funding}

\newcommand{\newblock}{\ }
\bibliographystyle{alpha}
\bibliography{bibliography}

\newcommand{\etalchar}[1]{$^{#1}$}
\begin{thebibliography}{DDHRZ20}

\bibitem[AL20]{alev2020improved}
Vedat~Levi Alev and Lap~Chi Lau.
\newblock Improved analysis of higher order random walks and applications.
\newblock In {\em Proceedings of the 52nd Annual ACM SIGACT Symposium on Theory
  of Computing}, pages 1198--1211, 2020.

\bibitem[ALG20]{anari2021spectral}
Nima Anari, Kuikui Liu, and Shayan~Oveis Gharan.
\newblock Spectral independence in high-dimensional expanders and applications
  to the hardcore model.
\newblock In Sandy Irani, editor, {\em 61st {IEEE} Annual Symposium on
  Foundations of Computer Science, {FOCS} 2020, Durham, NC, USA, November
  16-19, 2020}, pages 1319--1330. {IEEE}, 2020.

\bibitem[ALGV19]{anari2019log}
Nima Anari, Kuikui Liu, Shayan~Oveis Gharan, and Cynthia Vinzant.
\newblock Log-concave polynomials ii: high-dimensional walks and an fpras for
  counting bases of a matroid.
\newblock In {\em Proceedings of the 51st Annual ACM SIGACT Symposium on Theory
  of Computing}, pages 1--12, 2019.

\bibitem[BE21]{NE21}
Nikolas~P. Breuckmann and Jens~N. Eberhardt.
\newblock Balanced product quantum codes.
\newblock {\em IEEE Transactions on Information Theory}, 67:6653–6674, Oct
  2021.

\bibitem[CGM19]{cryan2019modified}
Mary Cryan, Heng Guo, and Giorgos Mousa.
\newblock Modified log-sobolev inequalities for strongly log-concave
  distributions.
\newblock In {\em 2019 IEEE 60th Annual Symposium on Foundations of Computer
  Science (FOCS)}, pages 1358--1370. IEEE, 2019.

\bibitem[CS96]{CSS}
A~Robert Calderbank and Peter~W Shor.
\newblock Good quantum error-correcting codes exist.
\newblock {\em Physical Review A}, 54(2):1098, 1996.

\bibitem[DD19]{dikstein2019agreement}
Yotam Dikstein and Irit Dinur.
\newblock Agreement testing theorems on layered set systems.
\newblock In {\em 2019 IEEE 60th Annual Symposium on Foundations of Computer
  Science (FOCS)}, pages 1495--1524. IEEE, 2019.

\bibitem[DDHRZ20]{dikstein2020locally}
Yotam Dikstein, Irit Dinur, Prahladh Harsha, and Noga Ron-Zewi.
\newblock Locally testable codes via high-dimensional expanders.
\newblock {\em arXiv preprint arXiv:2005.01045}, 2020.

\bibitem[DEL{\etalchar{+}}21]{dinur2021}
Irit Dinur, Shai Evra, Ron Livne, Alexander Lubotzky, and Shahar Mozes.
\newblock Locally testable codes with constant rate, distance, and locality.
\newblock {\em arXiv preprint arXiv:2111.04808}, 2021.

\bibitem[DK17]{dinur2017high}
Irit Dinur and Tali Kaufman.
\newblock High dimensional expanders imply agreement expanders.
\newblock In {\em 2017 IEEE 58th Annual Symposium on Foundations of Computer
  Science (FOCS)}, pages 974--985. IEEE, 2017.

\bibitem[DKW16]{DotterrerKW16}
Dominic Dotterrer, Tali Kaufman, and Uli Wagner.
\newblock On expansion and topological overlap.
\newblock In S{\'{a}}ndor~P. Fekete and Anna Lubiw, editors, {\em 32nd
  International Symposium on Computational Geometry, SoCG 2016, June 14-18,
  2016, Boston, MA, {USA}}, volume~51 of {\em LIPIcs}, pages 35:1--35:10.
  Schloss Dagstuhl - Leibniz-Zentrum f{\"{u}}r Informatik, 2016.

\bibitem[EK16]{EvraK16}
Shai Evra and Tali Kaufman.
\newblock Bounded degree cosystolic expanders of every dimension.
\newblock In {\em Proceedings of the 48th Annual {ACM} {SIGACT} Symposium on
  Theory of Computing, {STOC} 2016, Cambridge, MA, USA, June 18-21, 2016},
  pages 36--48. {ACM}, 2016.

\bibitem[EKZ20]{EvraKZ20}
Shai Evra, Tali Kaufman, and Gilles Z{\'{e}}mor.
\newblock Decodable quantum {LDPC} codes beyond the square root distance
  barrier using high dimensional expanders.
\newblock In {\em 61st {IEEE} Annual Symposium on Foundations of Computer
  Science, {FOCS} 2020, Durham, NC, USA, November 16-19, 2020}, pages 218--227,
  2020.

\bibitem[FK22]{FirstKaufman22}
Uriya~A. First and Tali Kaufman.
\newblock On good 2-query locally testable codes from sheaves on high
  dimensional expanders.
\newblock {\em CoRR}, abs/2208.01778, 2022.

\bibitem[FML02]{FML}
Michael~H Freedman, David~A Meyer, and Feng Luo.
\newblock Z2-systolic freedom and quantum codes.
\newblock {\em Mathematics of quantum computation, Chapman \& Hall/CRC}, pages
  287--320, 2002.

\bibitem[FN02]{friedland2002cheeger}
Shmuel Friedland and Reinhard Nabben.
\newblock On cheeger-type inequalities for weighted graphs.
\newblock {\em Journal of Graph Theory}, 41(1):1--17, 2002.

\bibitem[Gal]{GAL}
R.~G. Gallager.
\newblock Low density parity check codes.
\newblock {\em MIT Press, {\em Cambridge, MA, 1963}}.

\bibitem[Gar73]{garland1973p}
Howard Garland.
\newblock p-adic curvature and the cohomology of discrete subgroups of p-adic
  groups.
\newblock {\em Annals of Mathematics}, pages 375--423, 1973.

\bibitem[GK22a]{https://doi.org/10.48550/arxiv.2208.03241}
Roy Gotlib and Tali Kaufman.
\newblock Fine grained analysis of high dimensional random walks, 2022.

\bibitem[GK22b]{GotlibKaufman22}
Roy Gotlib and Tali Kaufman.
\newblock List agreement expansion from coboundary expansion, 2022.

\bibitem[Gol21]{golowich2021improved}
Louis Golowich.
\newblock Improved product-based high-dimensional expanders.
\newblock {\em arXiv preprint arXiv:2105.09358}, 2021.

\bibitem[Gro10]{Gro10}
Mikhail Gromov.
\newblock Singularities, expanders and topology of maps. part 2: From
  combinatorics to topology via algebraic isoperimetry.
\newblock {\em Geometric and Functional Analysis}, 20(2):416--526, 2010.

\bibitem[HHO]{HHO}
M.~B Hastings, J.~Haah, and R.~O'Donnell.
\newblock Fiber bundle codes: Breaking the n\({}^{\mbox{1/2}}\) polylog(n)
  barrier for quantum {LDPC} codes.
\newblock In {\em arXiv preprint arXiv:2009.03921, 2020, FOCS 2021 to appear.}

\bibitem[HLW06]{hoory06}
Shlomo Hoory, Nathan Linial, and Avi Wigderson.
\newblock Expander graphs and their applications.
\newblock {\em Bull. Amer. Math. Soc.}, 43(04):439–562, August 2006.

\bibitem[JVV86]{jerrum1986random}
Mark~R Jerrum, Leslie~G Valiant, and Vijay~V Vazirani.
\newblock Random generation of combinatorial structures from a uniform
  distribution.
\newblock {\em Theoretical computer science}, 43:169--188, 1986.

\bibitem[Kit03]{Toric}
A.~Y. Kitaev.
\newblock Fault-tolerant quantum computation by anyons.
\newblock {\em Ann. Phys.}, 303:2:2--30, 2003.

\bibitem[KKL16]{kaufman2016isoperimetric}
Tali Kaufman, David Kazhdan, and Alexander Lubotzky.
\newblock Isoperimetric inequalities for ramanujan complexes and topological
  expanders.
\newblock {\em Geometric and Functional Analysis}, 26(1):250--287, 2016.

\bibitem[KL14]{KaufmanLubotzky14}
Tali Kaufman and Alexander Lubotzky.
\newblock High dimensional expanders and property testing.
\newblock In {\em Proceedings of the 5th conference on Innovations in
  theoretical computer science}, pages 501--506, 2014.

\bibitem[KM17]{KaufmanM17}
Tali Kaufman and David Mass.
\newblock High dimensional random walks and colorful expansion.
\newblock In {\em 8th Innovations in Theoretical Computer Science Conference,
  {ITCS} 2017, January 9-11, 2017, Berkeley, CA, {USA}}, volume~67 of {\em
  LIPIcs}, pages 4:1--4:27, 2017.

\bibitem[KM20]{KaufmanM20}
Tali Kaufman and David Mass.
\newblock Local-to-global agreement expansion via the variance method.
\newblock In Thomas Vidick, editor, {\em 11th Innovations in Theoretical
  Computer Science Conference, {ITCS} 2020, January 12-14, 2020, Seattle,
  Washington, {USA}}, volume 151 of {\em LIPIcs}, pages 74:1--74:14, 2020.

\bibitem[KM21]{KaufmanM21}
Tali Kaufman and David Mass.
\newblock Unique-neighbor-like expansion and group-independent cosystolic
  expansion.
\newblock In {\em 32nd International Symposium on Algorithms and Computation,
  {ISAAC} 2021, December 6-8, 2021, Fukuoka, Japan}, volume 212 of {\em
  LIPIcs}, pages 56:1--56:17, 2021.

\bibitem[KO18]{kaufman2018high}
Tali Kaufman and Izhar Oppenheim.
\newblock High order random walks: Beyond spectral gap.
\newblock In {\em Approximation, Randomization, and Combinatorial Optimization.
  Algorithms and Techniques (APPROX/RANDOM 2018)}. Schloss
  Dagstuhl-Leibniz-Zentrum f{\"u}r Informatik, 2018.

\bibitem[KO21]{kaufman2021high}
Tali Kaufman and Izhar Oppenheim.
\newblock High dimensional expansion implies amplified local testability.
\newblock {\em arXiv preprint arXiv:2107.10488}, 2021.

\bibitem[KS08]{DBLP:conf/stoc/KaufmanS08}
Tali Kaufman and Madhu Sudan.
\newblock Algebraic property testing: the role of invariance.
\newblock In Cynthia Dwork, editor, {\em Proceedings of the 40th Annual {ACM}
  Symposium on Theory of Computing, Victoria, British Columbia, Canada, May
  17-20, 2008}, pages 403--412. {ACM}, 2008.

\bibitem[KT21a]{kaufman2021local}
Tali Kaufman and Ran~J Tessler.
\newblock Local to global high dimensional expansion and garland's method for
  general posets.
\newblock {\em arXiv preprint arXiv:2101.12621}, 2021.

\bibitem[KT21b]{KaufmanT21quantum}
Tali Kaufman and Ran~J. Tessler.
\newblock New cosystolic expanders from tensors imply explicit quantum {LDPC}
  codes with {\(\Omega\)}({\(\surd\)}\emph{n} log\({}^{\mbox{\emph{k}}}\)
  \emph{n}) distance.
\newblock In {\em {STOC} '21: 53rd Annual {ACM} {SIGACT} Symposium on Theory of
  Computing, Virtual Event, Italy, June 21-25, 2021}, pages 1317--1329, 2021.

\bibitem[LM06]{LM06}
Nathan Linial and Roy Meshulam.
\newblock Homological connectivity of random 2-complexes.
\newblock {\em Combinatorica}, 26(4):475--487, 2006.

\bibitem[LSV05]{lubotzky2005explicit}
Alexander Lubotzky, Beth Samuels, and Uzi Vishne.
\newblock Explicit constructions of ramanujan complexes of type $a_d$.
\newblock {\em European Journal of Combinatorics}, 26(6):965--993, 2005.

\bibitem[MV89]{mihail1989expansion}
Milena Mihail and Umesh Vazirani.
\newblock On the expansion of 0/1 polytopes.
\newblock {\em Journal of Combinatorial Theory. B}, 1989.

\bibitem[Opp18]{oppenheim2018local}
Izhar Oppenheim.
\newblock Local spectral expansion approach to high dimensional expanders part
  i: Descent of spectral gaps.
\newblock {\em Discrete \& Computational Geometry}, 59(2):293--330, 2018.

\bibitem[PK21]{PK21}
Pavel Panteleev and Gleb Kalachev.
\newblock Asymptotically good quantum and locally testable classical {LDPC}
  codes.
\newblock {\em CoRR}, abs/2111.03654, 2021.

\end{thebibliography}

\end{document}